\documentclass[a4paper,11pt,dvipsnames]{article}

\usepackage[T1]{fontenc}
\usepackage[utf8]{inputenc}
\usepackage[english]{babel}
\usepackage{lmodern}

\usepackage{amsthm, amssymb}
\usepackage{mathtools, thmtools}
\usepackage{dsfont, stmaryrd}

\usepackage[margin=3.5cm]{geometry}
\usepackage{graphicx, microtype}
\usepackage[labelfont=sc, size=footnotesize]{caption}
\usepackage{tikz, xcolor}

\usepackage[unicode, colorlinks=true,
pdfauthor={Ulysse Herbach},
pdftitle={Stochastic gene expression with a multistate promoter: breaking down exact distributions},
pdfsubject={Article},
bookmarksnumbered
]{hyperref}
\hypersetup{linkcolor=RoyalBlue, citecolor=Emerald, urlcolor=RubineRed}
\usepackage[nameinlink,capitalise,noabbrev]{cleveref}

\usepackage[pagestyles]{titlesec}
\usepackage{bookmark}

\newpagestyle{main}{
\sethead[][][]% even
{\ifthesection{\em\thesection\hspace{0.5em}\sectiontitle}{\em\sectiontitle}}{}{\ifthesubsection{\em\thesubsection\hspace{0.5em}\subsectiontitle}{\em\subsectiontitle}}% odd
\headrule
\setheadrule{0.5pt}
\setfoot{}{\thepage}{}
}

\renewpagestyle{plain}{\setfoot{}{\thepage}{}}

\declaretheorem[name=Theorem, numberwithin=section]{theorem}
\declaretheorem[name=Lemma, sibling=theorem]{lemma}
\declaretheorem[name=Proposition, sibling=theorem]{proposition}
\declaretheorem[name=Corollary, sibling=theorem]{corollary}
\declaretheorem[name=Remark, sibling=theorem, style=definition]{remark}

\definecolor{rouge}{rgb}{0.85,0.11,0.07}
\definecolor{vert}{rgb}{0.25,0.75,0.15}
\definecolor{bleu}{rgb}{0.0,0.36,0.80}
\definecolor{orange}{rgb}{0.94,0.43,0.02}
\definecolor{gris}{rgb}{0.5,0.5,0.5}

\newcommand*{\N}{\ensuremath \mathbb{N}} % N
\newcommand*{\Z}{\ensuremath \mathbb{Z}} % Z
\newcommand*{\R}{\ensuremath \mathbb{R}} % R
\renewcommand*{\C}{\ensuremath \mathbb{C}} % C
\newcommand*{\lb}{\ensuremath \llbracket} % Left bracket for integer segments
\newcommand*{\rb}{\ensuremath \rrbracket} % Right bracket for integer segments
\renewcommand*{\Re}{\ensuremath \operatorname{Re}} % Real part
\newcommand*{\segn}{\ensuremath {\llbracket 1, n \rrbracket}} % Integer [1,n]
\newcommand*{\dr}{\ensuremath \partial} % Partial derivative
\newcommand*{\intd}[1]{\ensuremath {d}{#1}} % differential
\newcommand*{\hyper}[2]{\ensuremath {{}_{#1}F_{#2}}} % Generalized hypergeometric function
\newcommand*{\Mat}[1]{\ensuremath \R^{#1\times #1}} % Square matrices of size n
\newcommand*{\diag}{\ensuremath \operatorname{Diag}} % Diagonal Matrix
\newcommand*{\Prob}{\ensuremath \mathbb{P}} % Probability
\newcommand*{\Esp}{\ensuremath \mathbb{E}} % Expectation
\newcommand*{\Var}{\ensuremath \operatorname{Var}} % Expectation
\newcommand*{\Poi}{\ensuremath \operatorname{Poisson}} % Poisson distribution
\newcommand*{\Gam}{\ensuremath \operatorname{Gamma}} % gamma distribution
\newcommand*{\Bet}{\ensuremath \operatorname{Beta}} % Beta distribution
\newcommand*{\Dir}{\ensuremath \operatorname{Dirichlet}} % Dirichlet distribution
\newcommand*{\chem}[1]{\ensuremath \mathsf{#1}} % Chemical species
\usetikzlibrary{cd}

\title{\vspace{-15mm}\textbf{Stochastic gene expression with a multistate promoter: breaking down exact distributions}}

\author{Ulysse Herbach\thanks{Univ Lyon, Université Claude Bernard Lyon 1, CNRS UMR 5208, Inria, Institut Camille Jordan, F-69603 Villeurbanne, France (\href{mailto:ulysse.herbach@inria.fr}{ulysse.herbach@inria.fr}).}}

\date{}

\begin{document}

\maketitle

\vspace{-5mm}

\begin{abstract}
We consider a stochastic model of gene expression in which transcription depends on a multistate promoter, including the famous two-state model and refractory promoters as special cases, and focus on deriving the exact stationary distribution.
Building upon several successful approaches, we present a more unified viewpoint that enables us to simplify and generalize existing results. In particular, the original jump process is deeply related to a multivariate piecewise-deterministic Markov process that may also be of interest beyond the biological field. In a very particular case of promoter configuration, this underlying process is shown to have a simple Dirichlet stationary distribution. In the general case, the corresponding marginal distributions extend the well-known class of Beta products, involving complex parameters that directly relate to spectral properties of the promoter transition matrix. Finally, we illustrate these results with biologically plausible examples.
\end{abstract}

\section{Introduction}
\label{sec:intro}

Gene expression within a cell, that is, transcription of specific regions of its DNA into mRNA molecules (to be then translated into proteins), is now well acknowledged to be a stochastic phenomenon, resulting from a set of various chemical reactions and typically modeled as a jump Markov process~\cite{Anderson2015}. Indeed, some of the chemical species involved are only present in very small quantities and their molecular fluctuations therefore generate a cellular “intrinsic noise”~\cite{Schnoerr2017}.
The entire set of underlying reactions may be huge but as a simple compromise, a gene is usually described by its \emph{promoter} and gene expression models consist of two reactions occurring in parallel: creation of mRNA by the promoter and degradation of mRNA~\cite{Dattani2017}. 
When both creation and degradation have constant rates, one gets a standard birth-death process that has a Poisson stationary distribution. Whereas such an elementary degradation is often accepted as a first-order approximation~\cite{Coulon2010,Schwanhausser2011}, the creation part is somewhat of a hot topic and more sophisticated models have been proposed, depending on the biological context~\cite{Coulon2010,Zoller2015,Dattani2017}.

For instance, measures of gene expression in individual, isogenic cells in the same environment typically show a heavy-tailed distribution with a clearly non-Poisson variance~\cite{Albayrak2016,Bahar-Halpern2015}. The simplest model to account for this fact is the well-established “two-state model”, which is a birth-death process in random environment~\cite{Peccoud1995,Kim2013}. As suggested by the name, such a promoter has one \emph{active} state in which the mRNA creation rate is positive and one \emph{inactive} state in which the creation rate is zero.
Depending on the switching rates between states, the time spent in the active one can be short enough to generate, when combined with a high creation rate, so-called “bursty” mRNA dynamics~\cite{Bahar-Halpern2015,Herbach2017} leading to a much more realistic distribution than the one-state previous model.

The two-state model has the advantage of being tractable, and sometimes it can even be physically justified as a relevant approximation (e.g., in bacteria~\cite{Chong2014}). However, as single-cell experiments become more precise, it appears that some promoters cannot be described by only two states because their inactive period has a nonexponential distribution with a positive mode~\cite{Zoller2015}. Such cases suggest a “refractory” behavior, meaning that after each active phase, the promoter has to progress through several inactive states before becoming active again.
Note that one could still consider a two-state model by accepting losing the Markov property. In fact, adding intermediate states is not only a convenient way to keep a Markov process, but also an opportunity to go into further details of the biological mechanisms behind gene expression. Though a complete description is still out of reach, such promoter states may indeed represent physical configurations of chromatin (the compacting molecular structure around DNA in eukaryotic cells), for example during remodeling steps~\cite{Coulon2013,Zoller2015}.

These observations have motivated the introduction of “multistate” promoters, each state being associated with a particular rate of mRNA creation~\cite{Coulon2010,Zhou2012,Innocentini2013,Dattani2017}.
Accordingly, we consider a promoter with $n$ states ($n\geqslant 2$) represented by chemical species $\chem{S}_1, \dots, \chem{S}_n$, with transitions between states such that molecule numbers always satisfy $[\chem{S}_i] \in \{0,1\}$ for all $i$ and $[\chem{S}_1] + \cdots + [\chem{S}_n] = 1$.
Then, representing mRNA by a species $\chem{M}$, the expression model is defined by the following system of elementary reactions:
%\vspace{4mm}
\begin{equation}\label{eq:chem_multistate}
\begin{cases}
\hspace{2mm}\chem{S}_i \xrightarrow[]{r_{i,j}} \chem{S}_j & \text{ for } i,j\in\lb 1,n\rb, \;i\neq j \vspace{1mm}  \\
\hspace{2mm}\chem{S}_i \xrightarrow[]{u_i} \chem{S}_i + \chem{M} & \text{ for } i\in\lb 1,n\rb \vspace{1mm} \\
\hspace{2mm}\chem{M} \xrightarrow[]{d_0} \varnothing &
\end{cases}
%\vspace{2mm}
\end{equation}
with rates $r_{i,j} \geqslant 0$, $u_i \geqslant 0$ and $d_0 > 0$.
Importantly, two distinct scenarios can be considered for this model:
\begin{enumerate}
\vspace{-2mm}
\item the general case (e.g.,~\cite{Coulon2010,Innocentini2013});
\vspace{-2mm}
\item the particular case where only one $u_i$ is nonzero (e.g.,~\cite{Zhou2012,Zoller2015,Dattani2017}).
\end{enumerate}\vspace{-2mm}
Promoters that belong to the second case can be interpreted as having exactly one active state and $n-1$ inactive states: in line with the intuition presented above, we shall call them \emph{refractory} in the present paper. Note that in this view, the two-state model corresponds to a “trivial” refractory promoter.

Our main interest here is the stationary distribution of the mRNA quantity $[\chem{M}]$.
In~\cite{Innocentini2013}, the authors provided a general but implicit formula based on a recurrence relation, focusing on multimodality induced by distinct $u_i$ values. On the other hand, the authors in~\cite{Zhou2012} consider some particular refractory promoters (transition graph forming a cycle) and express the distribution in terms of generalized hypergeometric functions~\cite{Slater1966}, providing an implicit way to derive the parameter values. A further step is achieved in~\cite{Dattani2016}, where parameters are explicitly derived in a more particular case (irreversible cycle).

In this paper, we propose to gather, simplify and extend these results by adopting a unified viewpoint: the underlying philosophy is to “break down the noise”, that is, to decompose the complexity of the distribution into simpler layers. As suggested in~\cite{Dattani2017}, we use the Poisson representation~\cite{Gardiner1977} of system~\eqref{eq:chem_multistate}, which allows for combining approaches in~\cite{Innocentini2013} and~\cite{Zhou2012} by introducing a piecewise-deterministic Markov process (PDMP).
First, we reinterpret the main result of~\cite{Innocentini2013} as a projection of the PDMP joint distribution. We show a simplistic situation where this distribution is Dirichlet, yet providing some interesting insight into the general case.
Second, we simplify the main result of~\cite{Zhou2012} concerning cyclic refractory promoters, and generalize it to any refractory promoter by only assuming irreducible dynamics (i.e., for any $i\neq j$, there exists a path of reactions from $\chem{S}_i$ to $\chem{S}_j$ with positive reaction rates).
This refractory case exactly corresponds to marginals of the previous joint distribution.
Interestingly, the resulting class of univariate distributions generalizes the one consisting of products of Beta-distributed random variables, which also arises in statistics~\cite{Dufresne2010} and mathematical physics~\cite{Dunkl2013}. It is characterized by a set of parameters that are potentially complex and directly relate to spectral properties of the promoter transition matrix.

The paper is organized as follows. The mathematical formulation of system~\eqref{eq:chem_multistate} is introduced in~\cref{sec:model} and its Poisson representation is detailed in~\cref{sec:poisson}. Then, the underlying multivariate PDMP is presented in~\cref{sec:multidim} and the complete solution for refractory promoters is given in \cref{sec:refractory}. Finally, applications are shown in~\cref{sec:applications} and a discussion follows in~\cref{sec:discussion}.

\section{Basic mathematical model}
\label{sec:model}

For $t\geqslant 0$, let $E_t$ and $M_t$ respectively denote the promoter state ($E_t = i$ if $[\chem{S}_i]=1$, $i\in\lb 1,n \rb$) and the mRNA level ($M_t = [\chem{M}]$) at time $t$.
Throughout this paper, we adopt a semivectorial notation by encoding promoter states as components of $\R^n$ while keeping mRNA as a scalar: this will make our computations much easier and will essentially reduce the results to linear algebra.
We assume that system~\eqref{eq:chem_multistate} follows standard \emph{stochastic mass-action} kinetics, that is, $(E_t,M_t)_{t\geqslant 0}$ is a jump Markov process with state space $\lb 1,n \rb \times \N$ and generator $\mathcal{L}$ defined by
\begin{equation}\label{eq:full_gen_backward}
\mathcal{L}f(k) = d_0k[f(k-1) - f(k)] + C[f(k+1) - f(k)] + Qf(k)
\end{equation}
where $f(k) = (f_1(k),\dots,f_n(k))^\top \in \R^n$ represents functions $f:\lb 1,n \rb \times \N \to \R$, $C = \diag(u_1,\dots,u_n) \in \Mat{n}$ contains creation rates and $Q \in \Mat{n}$ is the \emph{promoter transition matrix} given by
\[Q_{i,j} = r_{i,j} \quad \text{for} \;\; i\neq j \quad \text{and} \quad Q_{i,i} = -\sum_{j\neq i} r_{i,j} .\]

In practice, we shall focus on distributions (meaning probability measures here) and therefore consider the adjoint operator of $\mathcal{L}$, denoted by $\Omega$ and defined by
\begin{equation}\label{eq:full_gen_forward}
\Omega g(k) = d_0[(k+1)g(k+1) - kg(k)] + C[g(k-1)\mathds{1}_{k>0} - g(k)] + Hg(k)
\end{equation}
where $H=Q^\top$ and $g = (g_1,\dots,g_n)^\top$ now stands for distributions $g$ on $\lb 1,n \rb \times \N$. The distribution $p(t) = \Prob_{(E_t,M_t)}$, represented by $p(t) = (p_1(\cdot,t),\dots,p_n(\cdot,t))^\top$ where $p_i(k,t) = \Prob(E_t=i,M_t=k)$, then evolves according to the well-known Kolmogorov forward equation:
\begin{equation}\label{eq:full_master}
\frac{d p}{d t} = \Omega p
\end{equation}
which is often called \emph{master equation} in this context. Note that~\eqref{eq:full_master} is the same master equation as in~\cite{Coulon2010} and~\cite{Innocentini2013}. Also, it is a natural generalization of the master equation considered in~\cite{Zhou2012}, which corresponds here to \emph{cyclic refractory promoters} (i.e., only one $u_i$ is nonzero and the undirected graph induced by $H$ is a $n$-cycle).
See~\cref{sec:refractory} for a graphical representation of cyclic and general refractory promoters.

As mentioned above, we assume that $Q$ is irreducible (and thus also $H$): this is sufficient to ensure that $p(t)$ converges as $t\to\infty$ to a unique stationary distribution (see~\cite{Peccoud1995} and references therein), which will be our main object of interest.
Finally, we set $d_0 = 1$, say in $\mathrm{h}^{-1}$, without loss of generality (equivalent to dividing~\eqref{eq:full_gen_backward} and~\eqref{eq:full_gen_forward} by $d_0$ and rescaling time by $1/d_0$) so the model is completely parametrized by
\[r = (r_{i,j})_{i, j \in\lb 1,n \rb,\, i\neq j} \quad \text{and} \quad u = (u_1,\dots,u_n) .\]

\section{Poisson representation}
\label{sec:poisson}

In this section, we motivate the introduction of an underlying process that is not only useful for computations, but also arises naturally as a fundamental part of the original process $(E_t,M_t)_{t\geqslant 0}$. Our approach is based on the Poisson representation, initially introduced by Gardiner and Chaturvedi~\cite{Gardiner1977} as a powerful ansatz-based technique for solving master equations.
As emphasized by the authors, this representation is particularly adapted to chemical birth-death processes because of the particular jump rate form implied by stochastic mass-action kinetics~\cite{Gardiner1977,Petrosyan2014}, and it has already been successfully applied to intracellular kinetics~\cite{Thomas2010} and more specifically to gene expression~\cite{Iyer-Biswas2014,Sugar2014,Schnoerr2016}.
In our case, contrary to the original approach~\cite{Gardiner1977} in which all species are included, and as also done tacitly in~\cite{Iyer-Biswas2014,Sugar2014}, we shall apply the Poisson representation only to the mRNA part (species $\chem{M}$) and not to the promoter part (species $\chem{S}_i$, $i\in\segn$). The representation then becomes more than just an ansatz by revealing an actual “hidden layer” that turns out to be a piecewise-deterministic Markov process.

\subsection{Notation and definitions}
\label{seq:poisson:notation}

In all that follows, $\mathcal{M}(\R_+)$ and $\mathcal{M}(\N)$ denote the real vector spaces of finite signed measures on $\R_+ = [0,\infty)$ and $\N = \{0,1,2,\dots\}$.
When they are nonnegative and have their total mass equal to~$1$, such measures are standard probability measures, termed \emph{distributions} here.
It is worth mentioning that an actual, precise consideration of spaces is not the point of this article, but the reader may find some details in~\cref{sec:spaces}.
Note that $\mathcal{M}(\N)$ simply corresponds to the space of real sequences whose series is absolutely convergent.

Intuitively, $\mathcal{M}(\R_+)$ and $\mathcal{M}(\N)$ will describe mRNA levels.
Let us now introduce three transforms that will be used extensively in this paper.
Given $\mu\in\mathcal{M}(\R_+)$, we define the \emph{Laplace transform} $L_\mu$ by
\begin{equation}
\label{eq:def_laplace}
L_\mu(s) = \int_0^\infty e^{sx} \intd{\mu}(x), \quad \forall s\in(-\infty,0] .
\end{equation}
Similarly, given $p\in\mathcal{M}(\N)$, we consider the \emph{generating function} $G_p$ defined by
\begin{equation}
\label{eq:def_generating}
G_p(z) = \sum_{k=0}^\infty p(k) z^k, \quad \forall z\in[-1,1] .
\end{equation}
The last one is the starting point of the Poisson representation: for $\mu \in \mathcal{M}(\R_+)$, we define $P\mu \in \mathcal{M}(\N)$ by
\begin{equation}
\label{eq:def_poisson}
P\mu(k) = \int_0^\infty \frac{x^k}{k!}e^{-x}\intd{\mu}(x), \quad \forall k\in\N,
\end{equation}
which gives us a linear operator $P : \mathcal{M}(\R_+) \to \mathcal{M}(\N)$. We may term this operator the \emph{Poisson transform} and call its image $\mathcal{P}$ the space of \emph{Poisson mixtures}.
When $\mu$ has a density $f$ with respect to the Lebesgue measure on $\R_+$, we consistently set $Pf = P\mu$.
The operator $P$ clearly preserves total mass, that is, $P\mu(\N) = \mu(\R_+)$. Moreover, $\mu \geqslant 0$ implies $P\mu \geqslant 0$ and thus $P$ maps distributions to distributions (such operators are sometimes called \emph{stochastic}~\cite{Rudnicki2017}). In this case, the most important fact is that one can draw $Y\sim P\mu$ using the hierarchical (aka Bayesian) model:
\begin{equation*}
\begin{aligned}
X &\sim \mu \\
Y|X &\sim \Poi(X)
\end{aligned}
\end{equation*}
where the second line stands for the conditional distribution $\Prob_{Y|X} = \Poi(X)$.
In this sense, $P\mu$ is a \emph{mixture of Poisson distributions} where $\mu$ is the {mixing distribution}.

\begin{remark}\label{rem:laplace_c}
It will be useful to extend the definition of $L_\mu(s)$ to $s\in\C$. Given $\mu\in\mathcal{M}(\R_+)$, by standard results of complex analysis $L_\mu$ is holomorphic on the half plane $\{s\in\C \,|\, \Re(s)<0\}$. If $\mu$ is compactly supported (e.g., $\mu = \Prob_X$ with $X\in[0,1]$), then $L_\mu$ can be defined on $\C$ and is an entire function.
\end{remark}

The basic Poisson representation consists in finding an evolution equation for $\mu$ by assuming the form $P\mu$ in~\eqref{eq:full_master}, implicitly expecting $\mu$ to be simpler: this approach hence benefits from a remarkable probabilistic interpretation, in contrast to many other ansatz techniques commonly used to solve such equations.
Besides, the following result noted by Feller~\cite{Feller1943} enlightens the correspondence between the generating function and the Laplace transform in the context of Poisson mixtures.
Importantly for us, it implies that $P$ is injective.

\begin{lemma}\label{lem:feller}
Let $\mu\in\mathcal{M}(\R_+)$. Then for all $z\in[-1,1]$,
\[G_{P\mu}(z) = L_\mu(z-1) .\]
In particular, $P$ induces a linear isomorphism from $\mathcal{M}(\R_+)$ onto $\mathcal{P}$.
\end{lemma}

\begin{proof}
By Fubini's theorem (valid for $z\in[-1,1]$) applied on measures $\mu^+$ and $\mu^-$ of the Jordan decomposition $\mu = \mu^+ - \mu^-$, we directly get
\[G_{P\mu}(z) = \sum_{k=0}^\infty \int_0^\infty \frac{(zx)^k}{k!}e^{-x}\intd{\mu}(x) = \int_0^\infty e^{(z-1)x}\intd{\mu}(x) = L_\mu(z-1) .\]
Hence, if $P\mu = 0$, then $L_{\mu}(z) = 0$ for all $z\in[-2,0]$, and thus also for all $z\in(-\infty,0]$ by~\cref{rem:laplace_c}. As $\mu\mapsto L_\mu$ is injective on $\mathcal{M}(\R_+)$ (e.g.,~\cite[Theorem 15.6]{Klenke2014}), it follows that $\mu = 0$. Hence, $P$ is injective and we have $\mathcal{M}(\R_+) \simeq \mathcal{P}$.
\end{proof}

\begin{remark}\label{rem:poisson_gen}
Having this result in mind, it is no surprise that a very common way to solve master equations such as~\eqref{eq:full_master} consists in deriving an evolution equation for the generating function and then making the change of variable $s = z-1$. It seems to be less known that the outcome is nothing but the evolution equation satisfied by the Laplace transform of the mixing distribution within the Poisson representation.
\end{remark}

Finally, we just need to extend $\mathcal{M}(\R_+)$ and $\mathcal{M}(\N)$ in order to add a description of promoter states.
In line with the semivectorial definitions of $\mathcal{L}$ and $\Omega$ in~\eqref{eq:full_gen_backward} and~\eqref{eq:full_gen_forward}, we represent finite signed measures on $\lb 1,n \rb \times \R_+$ and $\lb 1,n \rb \times \N$ by
\[\mathcal{M}_n(\R_+) = \mathcal{M}(\R_+)^n \quad \text{and} \quad \mathcal{M}_n(\N) = \mathcal{M}(\N)^n .\]
The Laplace transform and the generating function are then naturally extended as
\[L_\mu = (L_{\mu_1},\dots,L_{\mu_n})^\top \quad \text{and} \quad G_p = (G_{p_1},\dots,G_{p_n})^\top \]
for $\mu=(\mu_1,\dots,\mu_n)^\top \in \mathcal{M}_n(\R_+)$ and $p=(p_1,\dots,p_n)^\top \in \mathcal{M}_n(\N)$, as well as the Poisson transform $P:\mathcal{M}_n(\R_+) \to \mathcal{M}_n(\N)$ whose image is now denoted by $\mathcal{P}_n$. Clearly, \cref{lem:feller} still holds and $P$ induces an isomorphism from $\mathcal{M}_n(\R_+)$ onto $\mathcal{P}_n$.

\subsection{Distribution viewpoint}
\label{sec:distrib:viewpoint}

The most intuitive way to derive the Poisson representation, assuming that $\mu$ admits a smooth density function of $y\in\R_+$, consists in simply injecting Poisson mixtures into equation~\eqref{eq:full_master} and then integrating by parts (see~\cref{sec:ansatz:approach}). If the boundary terms vanish, one finds that $p = P\mu$ is a solution of~\eqref{eq:full_master} if and only if
\begin{equation}\label{eq:pdmp_master}
\dr_t \mu = \dr_y(y\mu)- C\dr_y \mu + H\mu .
\end{equation}
This is the key idea in~\cite{Gardiner1977}: instead of using a series expansion, we obtain an \emph{exact} representation as a time-dependent mixture of Poisson distributions. In our case, the evolution equation~\eqref{eq:pdmp_master} satisfied by $\mu$ is a Kolmogorov forward equation associated with the new operator
\begin{equation}\label{eq:pdmp_gen_forward}
\widetilde{\Omega}\mu = -\dr_y[(C - yI)\mu] + H\mu ,
\end{equation}
which is the adjoint of
\begin{equation}\label{eq:pdmp_gen_backward}
\widetilde{\mathcal{L}} f = (C - yI)\dr_yf + Qf .
\end{equation}
A comparison of $\widetilde{\mathcal{L}}$ with~\eqref{eq:full_gen_backward} reveals that we made significant progress by going from a discrete to a continuous description. Notably, we thereby obtain the generator of a process $(E_t,Y_t)_{t\geqslant 0}$ that is a typical PDMP of state space $\lb 1,n \rb \times \R_+$, with $E_t$ being the same as before and $Y_t$ following the (random) differential equation
\begin{equation}\label{eq:ode_univariate}
\dot{Y}_t = \overline{u}(E_t) - Y_t
\end{equation}
where $\overline{u}(i) = u_i$ for $i\in\lb 1,n \rb$.
In other words, given the promoter state, the continuous variable $Y_t$ follows the traditional \emph{deterministic mass-action} kinetics of $\chem{M}$ in system~\eqref{eq:chem_multistate}: an example of such underlying PDMP is shown in~\cref{fig:sample_traj}.
The mixing distribution thus evolves exactly as we would expect when considering $[\chem{M}]$ continuous while keeping the $[\chem{S}_i]$ discrete, and indeed~\eqref{eq:pdmp_master}
can be rewritten as a simple system of coupled transport equations
\[\dr_t \mu + \dr_y[(C - yI)\mu] = H\mu\]
for which~\eqref{eq:ode_univariate} is the characteristic curve for each vector component.

\begin{figure}[tbp]
\begin{center}
\resizebox{\textwidth}{!}{\begin{tikzpicture}
\node at (0,0) {\includegraphics[width=\textwidth]{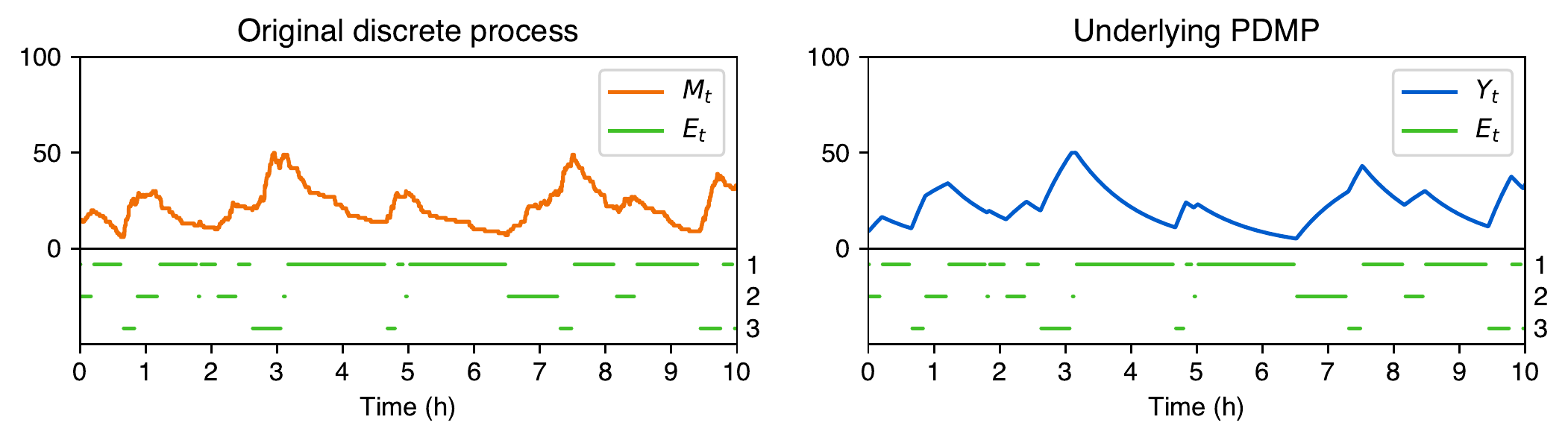}};
\node at (0.25,-3.7) {\resizebox{!}{3.3cm}{%
\begin{tikzpicture}[scale=0.9, line width = 1.2]
\tikzstyle{active}=[draw, circle, minimum width=9mm, color=gris, fill=gray!5]
\tikzstyle{inactive}=[draw, circle, minimum width=9mm, color=rouge]
\tikzstyle{bleft}=[bend left=13]
\tikzstyle{arrow}=[{->},>=latex]
\node[active] (S1) at (2.00,-0.00) {$\chem{S}_1$};
\node[active] (S2) at (-1.00,-1.73) {$\chem{S}_2$};
\node[active] (S3) at (-1.00,1.73) {$\chem{S}_3$};
\draw[arrow] (S1) to[bleft] node[midway, below right]{\color{black}\small $1$} (S2);
\draw[arrow] (S2) to[bleft] node[midway, above]{\color{black}\small $2$} (S1);
\draw[arrow] (S2) to[bleft] node[midway, left]{\color{black}\small $0.5$} (S3);
\draw[arrow] (S3) to[bleft] node[midway, right]{\color{black}\small $1$} (S2);
\draw[arrow] (S3) to[bleft] node[midway, above right]{\color{black}\small $2$} (S1);
\draw[arrow] (S1) to[bleft] node[midway, below]{} (S3);
\node (r13) at (0.14,0.35) {\small $0.5$};
\node[active] (S1) at (-9,1.45) {$\chem{S}_1$};
\node[active] (S2) at (-9,0) {$\chem{S}_2$};
\node[active] (S3) at (-9,-1.45) {$\chem{S}_3$};
\node[right] (e1) at (-8.2,1.45) {\large $\dot{Y}_t = - Y_t$};
\node[right] (e1) at (-8.2,0) {\large $\dot{Y}_t = 50 - Y_t$};
\node[right] (e1) at (-8.2,-1.45) {\large $\dot{Y}_t = 100 - Y_t$};
\end{tikzpicture}
}};
\end{tikzpicture}}
\caption{Sample paths of the original process and the underlying piecewise-deterministic process for an example of multistate promoter: $n=3$, $u=(0,50,100)$, $r_{2,1}=r_{3,1}=2$, $r_{1,2}=r_{3,2}=1$ and $r_{1,3}=r_{2,3}=0.5$. Here the paths of $(E_t,M_t)$ and $(E_t,Y_t)$ are generated using the same path of $E_t$.}
\label{fig:sample_traj}
\end{center}
\end{figure}
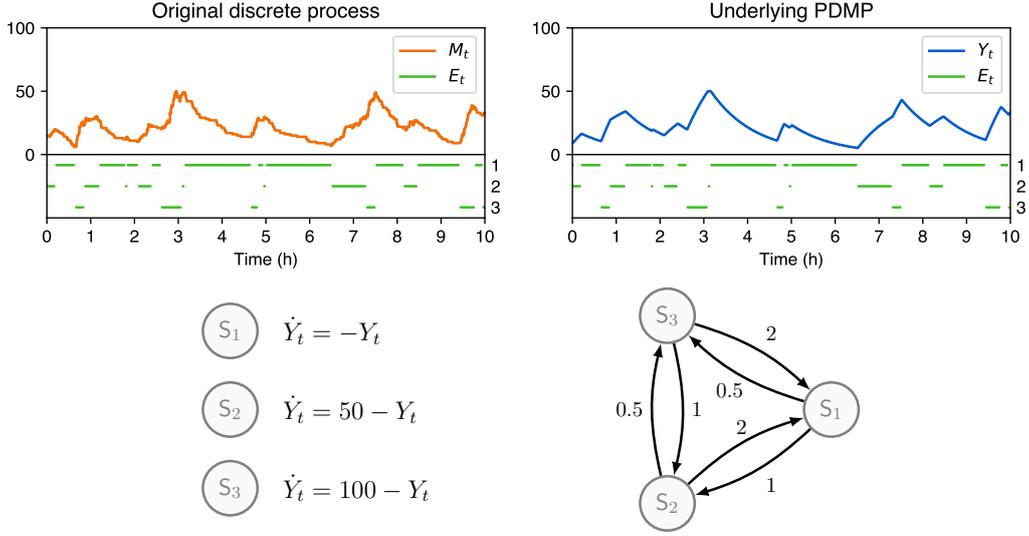

Although the previous method is only heuristic (boundary terms in the integration by parts may indeed not vanish), it is possible to get the same outcome rigorously using a dual approach related to~\cref{rem:poisson_gen}. More precisely, if $p(t)$ is the distribution of $(E_t,M_t)$ and $\mu(t)$ is the distribution of $(E_t,Y_t)$, then from the definition~\eqref{eq:full_gen_forward} of $\Omega$, the generating function $g(z,t) = G_{p(t)}(z)$ satisfies the evolution equation
\begin{equation}\label{eq:pde_generating}
\dr_t g(z,t) + (z-1)\dr_z g(z,t) = ((z-1)C+ H)g(z,t)
\end{equation}
while from the definition~\eqref{eq:pdmp_gen_forward} of $\widetilde{\Omega}$, the Laplace transform $\phi(s,t) = L_{\mu(t)}(s)$ satisfies
\begin{equation}\label{eq:pde_laplace}
\dr_t \phi(s,t) + s\dr_s \phi(s,t) = (sC + H)\phi(s,t) .
\end{equation}
Clearly, \eqref{eq:pde_generating} and~\eqref{eq:pde_laplace} perfectly coincide up to the change of variable $s = z - 1$. As a result (see~\cref{sec:dual:approach} for details), the dynamics of $p(t)$ coincide on the space of Poisson mixtures with the dynamics of $\mu(t)$ in the following sense.

\begin{proposition}\label{prop:stable_space}
Let $(S_t)_{t\geqslant 0}$ and $(\widetilde{S}_t)_{t\geqslant 0}$ be the operator semigroups generated by $\Omega$ and $\widetilde{\Omega}$, that is, for any $p_0\in \mathcal{M}_n(\N)$ and $\mu_0 \in \mathcal{M}_n(\R_+)$:
\begin{itemize}
\vspace{-2mm}
\item $p(t) = S_tp_0$ is the solution of~\eqref{eq:full_master} with initial condition $p(0) = p_0$;\vspace{-2mm}
\item $\mu(t) = \widetilde{S}_t\mu_0$ is the solution of~\eqref{eq:pdmp_master} with initial condition $\mu(0) = \mu_0$.
\end{itemize}\vspace{-2mm}
Then for all $t\geqslant 0$, the space of Poisson mixtures $\mathcal{P}_n \subset \mathcal{M}_n(\N)$ is an invariant subspace of $S_t$ and we have the following commutative diagram:
\[\begin{tikzcd}
\mathcal{M}_n(\R_+) \arrow{d}{P} \arrow{r}{\widetilde{S}_t} & \mathcal{M}_n(\R_+) \arrow{d}{P} \\
\mathcal{P}_n \arrow{r}{S_t} & \mathcal{P}_n
\end{tikzcd}\]
that is, $P^{-1} S_t P = \widetilde{S}_t$ where $P^{-1}$ is well defined on $\mathcal{P}_n$.
\end{proposition}

It is worth noticing that since $(E_t,M_t)_{t\geqslant 0}$ is ergodic~\cite{Peccoud1995}, its unique stationary distribution belongs to the space $\mathcal{P}_n$ which is therefore clearly a fundamental invariant subspace. It is not very common to know such a nontrivial subspace when dealing with infinite-dimensional semigroups, and this interesting result strongly suggests the introduction of the underlying process $(E_t,Y_t)_{t\geqslant 0}$. Unsurprisingly, it is also ergodic~\cite{Benaim2012} so the same Poisson representation holds in stationary regime.

\begin{corollary}\label{corol:stable_space}
The stationary distribution of the original process $(E_t,M_t)_{t\geqslant 0}$ is the Poisson mixture $p=P\mu$ where $\mu$ is that of $(E_t,Y_t)_{t\geqslant 0}$.
\end{corollary}

\begin{remark}
\label{rem:complex_measure}
Any chemical system leads to an analog of~\eqref{eq:pde_laplace}, but in general its solution is not the Laplace transform of a probability measure on $\R_+$ even when imposed at $t=0$ (see examples in~\cite{Gardiner1977}). Alternatively, one can always find a representation using a probability measure on $\C$, but then there is no straightforward analog of~\cref{prop:stable_space} since $P$ is not injective anymore~\cite[section 7.7.4]{Gardiner2004}.
\end{remark}

In our case, it is even possible to obtain a stronger result describing not only distributions but also sample paths: this approach appears in~\cite{Dattani2017} but we slightly adapt it here to our “invariant subspace” viewpoint.

\subsection{Path-based approach}
\label{sec:poisson:path}
In line with the chemical system~\eqref{eq:chem_multistate} and noticing that $(E_t)_{t\geqslant 0}$ is itself a jump Markov process with generator $Q$, one may alternatively consider $(M_t)_{t\geqslant 0}$ as a birth-death process in random environment $(E_t)_{t\geqslant 0}$, which can be described by a scalar, conditional master equation (see~\cref{sec:path:approach}).
The conditional generating function of mRNA given a promoter path is then defined by
\[\overline{g}(z,t) = \Esp\left[z^{M_t}|(E_\tau)_{\tau\geqslant 0}\right] = \Esp\left[z^{M_t}|(E_\tau)_{\tau\in[0,t]}\right] \]
and it satisfies the following partial differential equation:
\begin{equation}\label{eq:pde_generating_cond}
\dr_t \overline{g}(z,t) + (z-1)\dr_z \overline{g}(z,t) = (z-1)\overline{u}(E_t)\overline{g}(z,t) .
\end{equation}
This is just the analog of~\eqref{eq:pde_generating}, but much easier to solve since it is now a scalar transport equation. Using the standard method of characteristics, we get the following result.

\begin{proposition}\label{prop:generating}
Let $y_0 \in \R_+$. If $\overline{g}(z,0) = \exp(y_0(z-1))$, then we have
\[\overline{g}(z,t) = \exp(Y_t(z-1)) \]
for all $t\geqslant 0$, where
\[Y_t = e^{-t}y_0 +  \int_0^t \overline{u}(E_\tau) e^{\tau-t} \intd{\tau}\]
is the unique solution of the differential equation $\eqref{eq:ode_univariate}$ such that $Y_0 = y_0$.
\end{proposition}
Such $Y_t$ is well defined since $t\mapsto\overline{u}(E_t)$ is piecewise constant, and we can construct $(E_t,M_t)_{t\geqslant 0}$ and $(E_t,Y_t)_{t\geqslant 0}$ using the same path of $(E_t)_{t\geqslant 0}$ as in~\cref{fig:sample_traj}. In this case, if $M_0 \sim \Poi(y_0)$ is independent of $E_0$, then by~\cref{prop:generating},
\[\Prob_{M_t | (E_\tau)_{\tau\in[0,t]}} = \Prob_{M_t | Y_t}\]
and more specifically if $Y_0$ is independent of $E_0$ and $M_0|Y_0 \sim \Poi(Y_0)$, we get
\begin{equation}\label{eq:rep_pdmp}
\forall t \geqslant 0, \quad M_t | Y_t \sim \Poi(Y_t) .
\end{equation}
One must stay aware that the $M_t$ are not independent, even conditionally on $(Y_t)_{t\geqslant 0}$. However, we also obtain
$\Esp[(M_t)_{t\geqslant 0}|(E_t)_{t\geqslant 0}] = \Esp[(M_t)_{t\geqslant 0}|(Y_t)_{t\geqslant 0}] = (Y_t)_{t\geqslant 0}$
and thus, as clearly perceptible in~\cref{fig:sample_traj}, the whole path $(M_t)_{t\geqslant 0}$ can be interpreted as small Poisson-type fluctuations around $(Y_t)_{t\geqslant 0}$ which itself describes the core part of the dynamics.
This link between the two processes is in fact not specific to our choice of promoter dynamics: see~\cite{Dattani2017} for a more general presentation.

\section{Underlying multivariate structure}
\label{sec:multidim}

Following the Poisson representation, our interest is now the process $(Y_t)_{t\geqslant 0}$ defined by~\eqref{eq:ode_univariate}. In this section, we slightly change our point of view in order to reveal interesting symmetries.
More precisely, we introduce a multivariate process $(X_t)_{t\geqslant 0} = (X_{1,t}, \dots, X_{n,t})_{t\geqslant 0}$ with state space $\R_+^n = (\R_+)^n$, such that $X_{0}\in\Delta^{n-1}$, the $(n-1)$-simplex defined by
\[\Delta^{n-1} = \left\{x\in \R_+^n \,|\, x_1+\cdots+x_n = 1\right\} ,\]
and built from $(E_t)_{t\geqslant 0}$ so that when $E_t = i$,
\begin{equation}\label{eq:ode_multivariate}
\dot{X}_{j,t} = - X_{j,t} \; \text{ for } j\neq i \quad \text{and} \quad \dot{X}_{i,t} = \sum_{j\neq i} X_{j,t} .
\end{equation}
Regarding the original chemical formulation~\eqref{eq:chem_multistate}, these equations simply correspond to deterministic mass-action kinetics of species $\chem{X}_1,\dots,\chem{X}_n$ following reactions
\[\chem{S}_i + \chem{X}_j \xrightarrow[]{d_0} \chem{S}_i + \chem{X}_i \quad \text{ for } \, i,j\in\lb 1,n\rb, \; i\neq j ,\]
with $d_0=1$ in this article. The case $n=3$ is shown in~\cref{fig:multivariate}, with a sample path of $(E_t,X_t)_{t\geqslant 0}$ based on the same $(E_t)_{t\geqslant 0}$ as in~\cref{fig:sample_traj}.
In particular, it is easy to check from~\eqref{eq:ode_multivariate} that $X_{1,t} + \cdots + X_{n,t}$ is conserved, and thus $X_t\in\Delta^{n-1}$ for all $t\geqslant 0$.
The main point of introducing $X_t$ is the following result.

\begin{proposition}\label{prop_multivariate}
Let $u = (u_1,\dots,u_n) \in \R_+^n$ denote the mRNA creation rates. Then for all $t\geqslant 0$,
\begin{equation}\label{eq:rep_proj}
M_t | X_t \sim \Poi(u\cdot X_t)
\end{equation}
where $u\cdot X_t = u_1 X_{1,t} + \cdots + u_n X_{n,t}$, whenever $M_0 | X_0 \sim \Poi(u\cdot X_0)$.
\end{proposition}
\begin{proof}
Following~\eqref{eq:rep_pdmp}, we just need to show that $Y_t = u \cdot X_t$
whenever $Y_0 = u \cdot X_0$.
When $E_t=i$ and since $u$ is constant, the time derivative of $u \cdot X_t$ is
\[u_i \dot{X}_{i,t} + \sum_{j\neq i} u_j \dot{X}_{j,t} = u_i \sum_{j\neq i} X_{j,t} -\sum_{j\neq i} u_j X_{j,t} = u_i - u\cdot X_t\]
where we used~\eqref{eq:ode_multivariate} and the fact that $X_{1,t}+\cdots+X_{n,t} = 1$. The result immediately follows as $Y_t$ satisfies the same differential equation, that is, $\dot{Y}_t = u_i - Y_t$.
\end{proof}

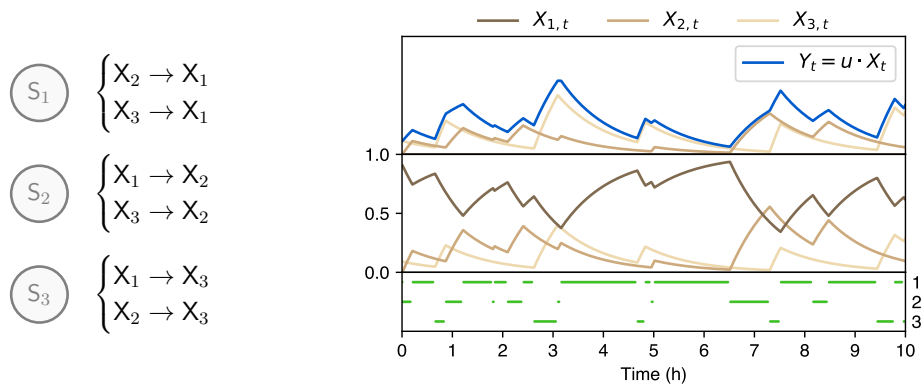
\begin{figure}[tbp]
\vspace{-4mm}
\begin{center}
\resizebox{\textwidth}{!}{\begin{tikzpicture}
\node at (0,0) {\includegraphics[width=\textwidth]{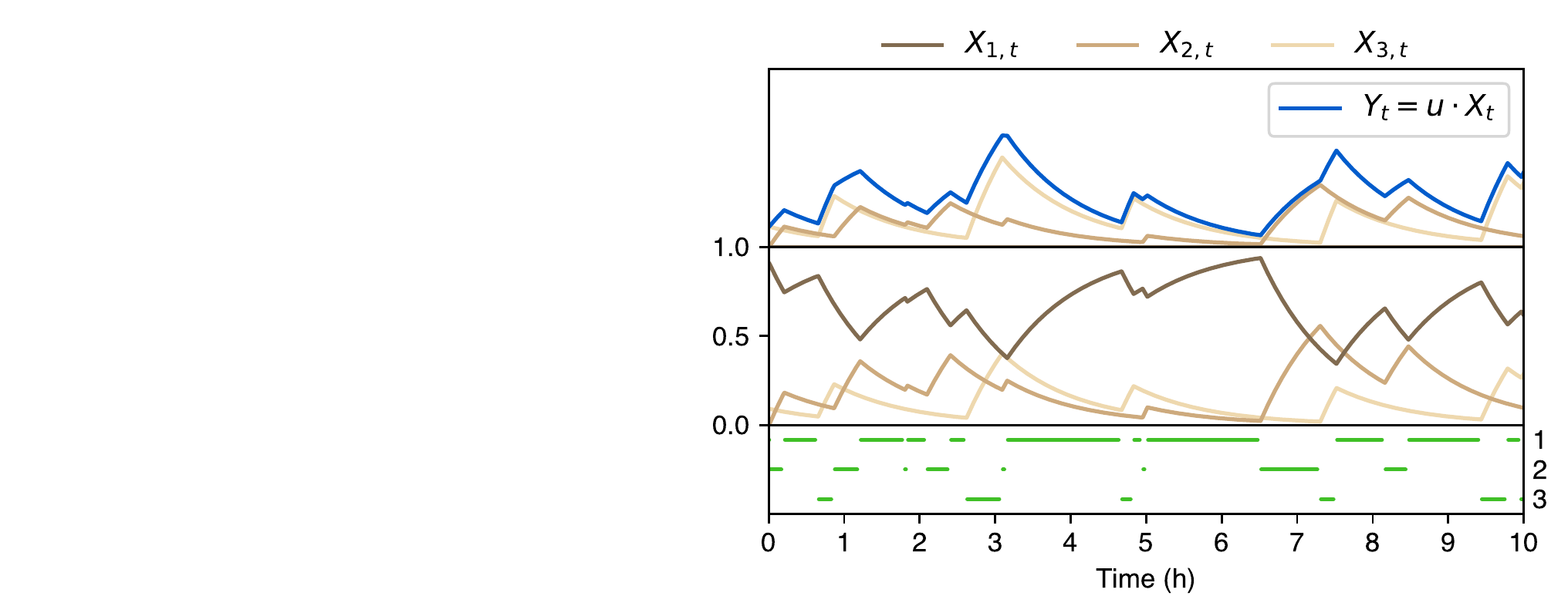}};
\node at (-3.75,0) {\resizebox{!}{4cm}{%
\begin{tikzpicture}[scale=0.9, line width = 1.2]
\tikzstyle{active}=[draw, circle, minimum width=9mm, color=gris, fill=gray!5]
\tikzstyle{inactive}=[draw, circle, minimum width=9mm, color=rouge]
\node[active] (S1) at (-9,1.8) {$\chem{S}_1$};
\node[active] (S2) at (-9,0) {$\chem{S}_2$};
\node[active] (S3) at (-9,-1.8) {$\chem{S}_3$};
\node[right] (e1) at (-8.2,1.8) {\large $\begin{cases}
\chem{X}_2 \to \chem{X}_1 \\
\chem{X}_3 \to \chem{X}_1
\end{cases}$};
\node[right] (e1) at (-8.2,0) {\large $\begin{cases}
\chem{X}_1 \to \chem{X}_2 \\
\chem{X}_3 \to \chem{X}_2
\end{cases}$};
\node[right] (e1) at (-8.2,-1.8) {\large $\begin{cases}
\chem{X}_1 \to \chem{X}_3 \\
\chem{X}_2 \to \chem{X}_3
\end{cases}$};
\end{tikzpicture}
}};
\end{tikzpicture}}\vspace{-3mm}
\caption{Multivariate PDMP in the three-state case. Left: chemical reactions associated with each promoter state. Right: sample path corresponding to the promoter example of~\cref{fig:sample_traj}, based on the same path of $E_t$. Using the same $u=(0,50,100)$ leads to the previous path of $Y_t = u \cdot X_t$.}
\label{fig:multivariate}
\end{center}
\end{figure}

Interestingly, the representation~\eqref{eq:rep_proj} can be interpreted as $Y_t$ being a “projection” of $X_t$ on $\R_+$ using $u$.
The initial condition in~\cref{prop_multivariate} turns out to be equivalent to $Y_0\in [\min(u),\max(u)]$ in~\eqref{eq:rep_pdmp}, which in fact is the physically relevant state space regarding the dynamics of $Y_t$ (i.e., values taken by $[\chem{M}]$ when treated as a concentration) so it is not too restrictive.
By~\cref{corol:stable_space}, the stationary distribution of $(M_t)_{t\geqslant 0}$ can always be represented as in~\eqref{eq:rep_proj} using that of $(X_t)_{t\geqslant 0}$.

Motivated by~\cref{prop_multivariate}, we shall focus on $(E_t,X_t)_{t\geqslant 0}$ which will be referred to as the “multivariate PDMP” as it is also an ergodic piecewise-deterministic Markov process~\cite{Benaim2012}.
Its generator is given from~\eqref{eq:ode_multivariate} by
\begin{equation}\label{eq:gen_multivariate_pdmp}
\widetilde{\mathcal{L}}_n f(x) = Qf(x) + \sum_{i=1}^n F_i(x) \dr_{x_i}f(x)
\end{equation}
for $x\in \R_+^n$, with $F_i(x) = (x_1+\dots+x_n)\diag(e_i) - x_iI$ where $\diag(e_i)\in \Mat{n}$ is the matrix whose only nonzero entry is $[\diag(e_i)]_{i,i}=1$.

\subsection{Multivariate Laplace transform}

Let $\mathcal{M}(\R_+^n)$ denote the vector space of finite measures on $\R_+^n$ and let $\mathcal{M}_n(\R_+^n) = \mathcal{M}(\R_+^n)^n$ represent finite measures on $\lb 1,n \rb \times \R_+^n$.
In the remainder of this article, we focus on the stationary distribution of the multivariate PDMP, denoted by $\mu = (\mu_1,\dots,\mu_n)^\top\in \mathcal{M}_n(\R_+^n)$, and consider a random variable $(E,X) \sim \mu$.
In line with our previous notation, we define the Laplace transform $\phi = L_\mu$ by
\begin{equation}\label{eq:def_laplace_multivariate}
\phi_i(s) = \int_{\R_+^n} e^{s\cdot x} \intd{\mu_i}(x) = \Esp\left[\mathds{1}_{\{E=i\}} e^{s\cdot X}\right], \quad \forall s\in (-\infty,0]^n .
\end{equation}
From~\eqref{eq:gen_multivariate_pdmp} and the fact that $X \in\Delta^{n-1}$ (which is equivalent to $\dr_{s_1}\phi + \cdots + \dr_{s_n}\phi = \phi$), we find that $\phi$ satisfies
\begin{equation}\label{eq:laplace_multivariate}
\sum_{i=1}^n s_i\dr_{s_i}\phi(s) = (D(s) + H)\phi(s)
\end{equation}
with $D(s) = \diag(s_1,\dots,s_n)$.

Besides, an analog of~\cref{rem:laplace_c} implies that $\phi$ can be extended to an entire function on $\R^n$. More precisely, we have
\begin{equation}\label{eq:series_moments}
\phi(s) = \sum_{\alpha \in \N^n} m(\alpha) \frac{s_1^{\alpha_1} \cdots s_n^{\alpha_n}}{\alpha_1 ! \cdots \alpha_n !}
\end{equation}
where $m(\alpha)$ is given, using notation $\dr_s^\alpha = \dr_{s_1}^{\alpha_1} \cdots \dr_{s_n}^{\alpha_n}$ and $X^\alpha = {X_1}^{\alpha_1} \cdots {X_n}^{\alpha_n}$, by
\begin{equation}\label{eq:def_mellin}
m_i(\alpha) = \dr_s^\alpha \phi_i(0) = \Esp\left[\mathds{1}_{\{E=i\}} X^\alpha\right] .
\end{equation}
The convergence of the series~\eqref{eq:series_moments} for all $s\in\R^n$ is then immediate as $m_i(\alpha)\in[0,1]$ for all $\alpha \in \N^n$, by~\eqref{eq:def_mellin} and the fact that $X\in\Delta^{n-1}\subset[0,1]^n$.

\subsection{General recursion formula}

We are now interested in solving system~\eqref{eq:laplace_multivariate}. Given some fixed $s \in \R^n$, the change of unknown $\Phi_s(\omega) = \phi(\omega s_1,\dots,\omega s_n) = \phi(\omega s)$ for $\omega\in\R$ directly leads to a much simpler one-variable problem.
\begin{proposition}\label{prop:laplace_ode}
For all $s\in\R^n$, $\omega \mapsto \Phi_s(\omega)$ statisfies
\begin{equation}\label{eq:laplace_proj}
\omega\frac{d\Phi_s}{d\omega}(\omega) = (\omega D(s) + H)\Phi_s(\omega)
\end{equation}
which is an ordinary differential system.
\end{proposition}

It should be noted from~\eqref{eq:def_laplace_multivariate} that $\Phi_s$ is in fact the Laplace transform of $(E,Y)$ where $Y = s \cdot X$ (i.e., $\Phi_s = L_{\overline{\mu}}$ with $(E,Y)\sim\overline{\mu}$). Combined with~\cref{lem:feller}, this explains why~\eqref{eq:laplace_proj} happens to coincide, in the special case $s = u$, with the main equation considered in~\cite{Innocentini2013}.
An important consequence of~\eqref{eq:series_moments} is that $\Phi_s$ can be expressed as a power series in $\omega$, and some simple computation from~\eqref{eq:laplace_proj} then leads to the following recurrence relation between the coefficients.

\begin{corollary}\label{corol:rec_laplace}
Let $s\in\R^n$. Then $\Phi_s(\omega) = \sum_{k=0}^\infty c_k(s) \omega^k$ where
\begin{equation}\label{eq:rec_laplace}
Hc_0(s) = 0 \quad \text{and} \quad \forall k\geqslant 1, \quad c_k(s) = (kI-H)^{-1}D(s) c_{k-1}(s) .
\end{equation}
\end{corollary}

This recursion formula corresponds to the one used in~\cite{Innocentini2013}.
The irreducibility of $H$ is crucial here: a classic application of the Perron--Frobenius theorem shows that eigenvalues of $H$ all have negative real parts except the simple eigenvalue 0 (see \cref{sec:spectral_properties}), so the recurrence~\eqref{eq:rec_laplace} is well defined and $(c_k(s))_{k\geqslant 0}$ is unique up to a multiplicative constant.
Taking $\omega = 1$, we finally obtain
\begin{equation}\label{eq:series_simple}
\phi(s) = \sum_{k=0}^\infty c_k(s)
\end{equation}
which can be seen as a particular choice of summation in~\eqref{eq:series_moments}.
Since the distribution of $E$ is by definition $c_0(s) = c_0 = m(0)$, the sequence $(c_k(s))_{k\geqslant 0}$ is unique, confirming the uniqueness of $\mu$ under the assumption that $X_0 \in \Delta^{n-1}$ almost surely.

Although useful for numerical computation, especially when $H$ is diagonalizable, the recurrence relation~\eqref{eq:rec_laplace} does not make $\phi$ really explicit, the main challenge being that matrices $D(s)$ and $H$ do not commute.
Remarkably, the case where only one $s_i$ is nonzero turns out to be explicitly solvable: it is the object of~\cref{sec:refractory}. Before, let us however present a fully solvable promoter configuration.

\subsection{A fully solvable case}
\label{sec:dirichlet}

Let $\alpha_1,\dots,\alpha_n$ be positive parameters and consider the very particular case
\begin{equation}\label{eq:dirichlet_model}
r_{i,j} = \alpha_j , \quad \forall i,j\in\lb 1,n\rb, \;i\neq j .
\end{equation}
Namely, each promoter state $\chem{S}_i$ can be reached directly from all other states $\chem{S}_j$ with the same rate $\alpha_i$.
As an example, the three-state promoter in~\cref{fig:sample_traj} and thus the multivariate sample path in~\cref{fig:multivariate} belong to this case.
Such dynamics clearly have little memory: this situation is simplistic from a biological perspective but may be viewed as a useful first step, giving some interesting insight into $\mu$ in general.

\begin{proposition}\label{prop:dirichlet_promoter}
If the promoter transition rates satisfy~\eqref{eq:dirichlet_model}, the stationary distribution of mRNA is given by the hierarchical model
\begin{equation*}
\begin{aligned}
X &\sim \Dir(\alpha_1,\dots,\alpha_n) \\
M | X &\sim \Poi\left(u_1 X_1 + \cdots + u_n X_n\right)
\end{aligned}
\end{equation*}
where $X$ corresponds to the multivariate PDMP and $M$ corresponds to mRNA.
\end{proposition}

\begin{proof}
In this easy case it is possible to use~\cref{corol:rec_laplace} but not even necessary. Indeed, we obtain from~\eqref{eq:gen_multivariate_pdmp} that $\mu$ satisfies the stationary equation
\begin{equation}\label{eq:pdmp_forward_multivariate}
\sum_{i=1}^n \dr_{x_i}(F_i \mu) = H\mu
\end{equation}
with $F_i(x) = (x_1+\dots+x_n)\diag(e_i) - x_iI$ and $H$ derived from~\eqref{eq:dirichlet_model}, that is:
\[\sum_{j\neq i}\dr_{x_i}(x_j\mu_i) - \dr_{x_j}(x_j\mu_i) = \alpha_i\sum_{j\neq i} \mu_j - \mu_i \sum_{j\neq i} \alpha_j , \quad \forall i\in\lb 1,n \rb .\]
This system turns out to be solvable “piece by piece”.
Let $\sigma$ denote the induced Lebesgue measure on $\Delta^{n-1}$, and let $\mu\in\mathcal{M}_n(\R_+^n)$ be defined by the density $f=(f_1,\dots,f_n)^\top$ with respect to $\sigma$ (meaning $\mu$ has all its mass in $\Delta^{n-1}$) up to a normalizing constant, where
\[f_i(x) = x_i\prod_{j=1}^n x_j^{\alpha_j-1} , \quad \forall i\in\lb 1,n \rb .\]
Then we easily get $\dr_{x_i}(x_jf_i) = \alpha_i f_j$ and $\dr_{x_j}(x_jf_i) = \alpha_j f_i$ for all $i,j\in\lb 1,n\rb$, $i\neq j$, and it follows that $\mu$ is solution of~\eqref{eq:pdmp_forward_multivariate} in the weak sense.
In other words, we have
\[\Prob_{X|E=i} = \Dir(\alpha_1,\dots,\alpha_{i-1},\alpha_i+1,\alpha_{i+1},\dots,\alpha_n), \quad \forall i\in\lb 1,n \rb \]
and marginalizing over $E$ leads to $\Prob_{X} = \mu_1 + \cdots + \mu_n = \Dir(\alpha_1,\dots,\alpha_n)$.
\end{proof}

\section{Complete solution for refractory promoters}
\label{sec:refractory}

Recall that a promoter state $i\in\lb 1,n \rb$ is \emph{active} if $u_i > 0$ and \emph{inactive} if $u_i = 0$. In this section we consider the particular case of \emph{refractory promoters}, that is, for which only one state is active.
In line with the previous section, we consider a random variable $M$ generated by
\[M | X \sim \Poi(u\cdot X)\]
so $\Prob_M$ is the mRNA stationary distribution.
Without loss of generality, we assume that the active state is the first one (\cref{fig:refractory}) with $u_1 = \nu >0$ so that $u \cdot X = \nu X_1$.
We derive an explicit formula for $\Prob_M$ in this case, thereby simplifying and generalizing the results in~\cite{Zhou2012}, and extend some of the ideas in~\cite{Dattani2016,Dattani2017} concerning $X_1$.

\bigskip

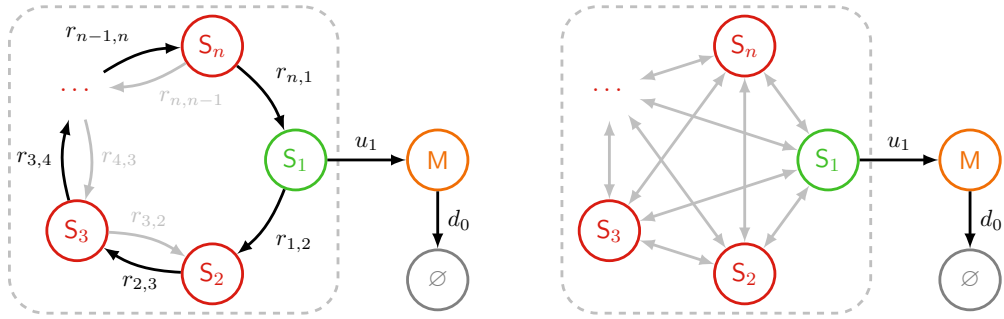
\begin{figure}[htbp]
\begin{center}
\resizebox{!}{4.1cm}{%
\begin{tikzpicture}[scale=0.9, line width = 1.2]
\tikzstyle{active}=[draw, circle, minimum width=9mm, color=vert]
\tikzstyle{inactive}=[draw, circle, minimum width=9mm, color=rouge]
\tikzstyle{RNA}=[draw, circle, minimum width=9mm, color=orange]
\tikzstyle{RNAx}=[draw, circle, minimum width=9mm, color=gris]
\tikzstyle{bleft}=[bend left=15]
\tikzstyle{arrow}=[{->},>=latex]
\tikzstyle{gray}=[color=gray!50]
\node[draw, color=gray!50, dashed, rounded corners=15pt, minimum width=48.95mm, minimum height=46mm] (rect1) at (-0.02,0) {};
\node[active]   (S1) at (2.00,-0.00) {$\chem{S}_1$};
\node[inactive] (S2) at (0.62,-1.90) {$\chem{S}_2$};
\node[inactive] (S3) at (-1.62,-1.18) {$\chem{S}_3$};
\node[color=rouge, circle, minimum width=9mm] (S4) at (-1.62,1.18) {\dots};
\node[inactive] (S5) at (0.62,1.90) {$\chem{S}_n$};
\draw[arrow] (S1) to[bleft] node[midway, color=black, below right]{\small $r_{1,2}$} (S2);
\draw[arrow] (S5) to[bleft] node[midway, color=black, above right]{\small $r_{n,1}$} (S1);
\draw[arrow] (S2) to[bleft] node[midway, color=black, below]{\small $r_{2,3}$} (S3);
\draw[arrow,gray] (S3) to[bleft] node[midway, above, gray]{\small $r_{3,2}$} (S2);
\draw[arrow] (S3) to[bleft] node[midway, color=black, left]{\small $r_{3,4}$} (S4);
\draw[arrow, gray] (S4) to[bleft] node[midway, right, gray]{\small $r_{4,3}$} (S3);
\draw[arrow] (S4) to[bleft] node[midway, color=black, above left]{\small $r_{n-1,n}$} (S5);
\draw[arrow, gray] (S5) to[bleft] node[midway, below right, gray]{\small $r_{n,n-1}$} (S4);
\node[RNA] (M) at (4.35,0) {$\chem{M}$};
\node[RNAx] (Mx) at (4.35,-2) {$\varnothing$};
\draw[arrow] (S1) to node[midway, color=black, above]{\small $u_1$} (M);
\draw[arrow] (M) to node[midway, color=black, right]{\small $d_0$} (Mx);
\end{tikzpicture}
} \hspace{7mm}
\resizebox{!}{4.1cm}{%
\begin{tikzpicture}[scale=0.9, line width = 1.2]
\tikzstyle{active}=[draw, circle, minimum width=9mm, color=vert]
\tikzstyle{inactive}=[draw, circle, minimum width=9mm, color=rouge]
\tikzstyle{RNA}=[draw, circle, minimum width=9mm, color=orange]
\tikzstyle{RNAx}=[draw, circle, minimum width=9mm, color=gris]
\tikzstyle{bleft}=[bend left=0]
\tikzstyle{gray}=[color=gray!50]
\tikzstyle{arrow}=[{->},>=latex]
\tikzstyle{link}=[{<->},>=latex, gray]
\node[draw, color=gray!50, dashed, rounded corners=15pt, minimum width=46mm, minimum height=46mm] (rect2) at (0.15,0) {};
\node[active]   (S1) at (2.00,-0.00) {$\chem{S}_1$};
\node[inactive] (S2) at (0.62,-1.90) {$\chem{S}_2$};
\node[inactive] (S3) at (-1.62,-1.18) {$\chem{S}_3$};
\node[color=rouge, circle, minimum width=9mm] (S4) at (-1.62,1.18) {\dots};
\node[inactive] (S5) at (0.62,1.90) {$\chem{S}_n$};
\draw[link] (S1) to[bleft] node[color=black]{} (S2);
\draw[link] (S1) to node[color=black]{} (S3);
\draw[link] (S1) to node[color=black]{} (S4);
\draw[link] (S5) to[bleft] node[color=black]{} (S1);
\draw[link] (S2) to[bleft] node[color=black]{} (S3);
\draw[link] (S2) to node[color=black]{} (S4);
\draw[link] (S2) to node[color=black]{} (S5);
\draw[link] (S3) to[bleft] node[color=black]{} (S4);
\draw[link] (S3) to node[color=black]{} (S5);
\draw[link] (S4) to[bleft] node[color=black]{} (S5);
\node[RNA] (M) at (4.35,0) {$\chem{M}$};
\node[RNAx] (Mx) at (4.35,-2) {$\varnothing$};
\draw[arrow] (S1) to node[midway, color=black, above]{\small $u_1$} (M);
\draw[arrow] (M) to node[midway, color=black, right]{\small $d_0$} (Mx);
\end{tikzpicture}
}
\caption{Dynamics of the system in the case of refractory promoters. Black arrows correspond to reactions with positive rates while grey ones indicate reactions that can have rate $0$.
Left: cyclic refractory promoters as considered in~\cite{Zhou2012}. Right: general refractory promoters as considered here.}
\label{fig:refractory}
\end{center}
\end{figure}

\begin{remark}
It should now be clear that refractory promoters are associated with marginal distributions of $X$. For instance, one easily recovers the fact that $\Prob_M$ is a scaled Beta-Poisson mixture in the case of the two-state model~\cite{Kim2013}: when $n=2$, $r_{2,1} = \alpha_1$ and $r_{1,2} = \alpha_2$, we get $X \sim \Dir(\alpha_1,\alpha_2)$ by \cref{prop:dirichlet_promoter} and then it is well known that $X_1 \sim \Bet(\alpha_1,\alpha_2)$ and $X_2 \sim \Bet(\alpha_2,\alpha_1)$.
\end{remark}

\subsection{Notation and definitions}

The refractory case also provides notions of active and inactive \emph{periods}, which are respectively the time $T_1$ spent in the active state before becoming inactive and the time $T_0$ spent in other states before becoming active.
By definition of the process, $T_1$ has density $f_{T_1}$ on $[0,+\infty)$ given by $f_{T_1}(t) = \lambda \exp(-\lambda t)$ where $\lambda=\sum_{j\neq 1} r_{1,j}$, and the density of $T_0$ is also available explicitly.

\begin{lemma}\label{lem:inactive_period}
The inactive period $T_0$ has density $f_{T_0}$ defined on $[0,+\infty)$ by
\[f_{T_0}(t) = \left[\tilde{H}\exp(t\tilde{H})\pi\right]_1\]
where $\pi = (\pi_1,\dots,\pi_n)^\top$ is defined by $\pi_{1} = 0$ and $\pi_i = r_{1,i}/{\sum_{j\neq 1} r_{1,j}}$ for $i\neq 1$, and $\tilde{H}\in\Mat{n}$ is obtained from $H$ by replacing the first column with zeros.
\end{lemma}
\begin{proof}
Consider the Markov process $(\tilde{E}_t)_{t\geqslant 0}$ with generator $\tilde{H}$ (it becomes “stuck” in state $i=1$ as soon as it is reached) and such that $\Prob_{\tilde{E}_0} = \pi$. The inactive period can then be defined as $T_0 = \inf\{t>0 \,|\, \tilde{E}_t = 1 \}$. Its cumulative distribution function is
\[F_{T_0}(t) = \Prob(T_0 \leqslant t) = \Prob(\tilde{E}_t = 1) = \left[\exp(t\tilde{H})\pi\right]_1 \]
and the result follows by taking the derivative of $F_{T_0}$.
\end{proof}
The distribution of $T_0$ is not the main point here but it enlightens the underlying linear algebra that also appears in the next results. In addition, we shall use~\cref{lem:inactive_period} in~\cref{sec:applications} to gain insight into the particular dynamics of some promoters.

As found in~\cite{Zhou2012,Dattani2017}, distributions of $M$ and $X_1$ can be expressed in a compact way using generalized hypergeometric functions. Let us introduce some related notation, borrowed from~\cite{Slater1966}. Given $A\in\N$ and $a = (a_1,\dots,a_{A}) \in \C^{A}$, we define
\[(a)_k = \prod_{i=1}^A \frac{\Gamma(a_i+k)}{\Gamma(a_i)} = \prod_{i=1}^A a_i (a_i+1) \cdots (a_i+k-1)\]
for $k\in\N$, adopting the convention $(a)_k = 1$ if $A=0$. Then, given also $B\in\N$ and $b = (b_1,\dots,b_B) \in \C^B$, we define the hypergeometric function $\hyper{A}{B}$ as
\[\hyper{A}{B}[(a);(b);x] = \sum_{k=0}^\infty \frac{(a)_k}{(b)_k} \frac{x^k}{k!} \]
for $x\in\R$ such that the series is well defined and converges.
For $z\in\C$, we shall write $a+z = (a_1+z,\dots,a_A+z)\in\C^A$ and $b+z = (b_1+z,\dots,b_B+z)\in\C^B$.

We finally set $n = N+1$ to simplify the notation in this section, so that $N\geqslant 1$ is the number of inactive states.
Then, combining the Perron--Frobenius theorem as mentioned beside \cref{corol:rec_laplace} with some other linear algebra results (see~\cref{sec:spectral_properties}), we can define two fundamental families of eigenvalues.

\begin{lemma}\label{lem:eigenvalues}
For $i\in\lb 1,n \rb$, let $H_{\{i\}}\in\Mat{N}$ be the matrix obtained from $H$ by removing the $i$-th row and the $i$-th column. Furthermore, let
\vspace{-1mm}
\[a^{(i)} = (a^i_1,\dots,a^i_N)\in\C^N \quad \text{and} \quad b = (b_1,\dots,b_N)\in\C^N\vspace{-1mm}\]
respectively denote the eigenvalues of $-H_{\{i\}}$ and the nonzero eigenvalues of $-H$, counted with multiplicity. Then $a^i_1,\dots,a^i_N$, $b_1,\dots,b_N$ all have positive real parts and the promoter stationary distribution is given by
\vspace{-2mm}
\[\Prob_E(i) = \prod_{k=1}^N \frac{a^i_k}{b_k}, \quad \forall i\in\lb 1,n \rb .\]
\end{lemma}
Remarkably, it turns out that $\Prob_M$ is directly parametrized by these eigenvalues.

\subsection{Exact mRNA distribution}

Let us begin by describing the continuous component $X_1$.
In all that follows, we consider $a = a^{(1)}$ and $b$ as defined in \cref{lem:eigenvalues}.
The results will be based on $X_1$ for simplicity but immediately generalize to any $X_i$ by replacing $a = a^{(1)}$ with $a^{(i)}$.

\begin{theorem}
The Laplace transform of $X_1$ is given by
\begin{equation}\label{eq:laplace_marginal}
L_{X_1}(s_1) = \Esp\left[e^{s_1 X_1}\right] = \hyper{N}{N}[(a);(b);s_1] .
\end{equation}
\end{theorem}
\begin{proof}
The idea is to solve the recurrence relation~\eqref{eq:rec_laplace} to get $c_k(s)$ for all $k\in\N$, assuming that $s = (s_1,0,\dots,0)$. Since we marginalize over the promoter state $E$, we are only interested in $\overline{c}_k = c_{k,1}+\cdots+c_{k,n}$.
First, summing vectorial components in $(kI-H) c_k(s) = D(s) c_{k-1}(s)$ yields $k\overline{c}_k(s) = s_1 c_{k-1,1}(s)$ so we just need $c_{k,1}(s)$. Second, it is straightforward to see that the $(1,1)$-cofactor of $kI-H$ is equal to $\det(kI - H_{\{1\}}) = (a_1 + k) \cdots (a_N + k)$ and we use it in~\eqref{eq:rec_laplace} through the standard inversion formula to get the scalar recurrence relation
\[c_{k,1}(s) = \frac{\det(kI-H_{\{1\}})}{\det(kI-H)} s_1 c_{k-1,1}(s) = \frac{(a_1 + k) \cdots (a_N + k)}{k (b_1 + k) \cdots (b_N + k)} s_1 c_{k-1,1}(s) , \quad \forall k\geqslant 1 .\]
The initial term is $c_{0,1} = \Prob_E(1) = (a_1 \cdots a_N)/(b_1 \cdots b_N)$ by~\cref{lem:eigenvalues}. Finally,
\[\overline{c}_k(s) = \frac{s_1}{k} c_{k-1,1}(s) = \frac{(a)_{k}}{(b)_{k}} \frac{s_1^k}{k!} , \quad \forall k \in \N\]
and the result directly follows from~\eqref{eq:series_simple}.
\end{proof}

Computing the derivatives of $s_1 \mapsto L_{X_1}(\nu s_1)$ and using~\cref{lem:feller}, we obtain the mRNA stationary distribution for general refractory promoters.

\begin{corollary}
If $u = (\nu,0,\dots,0)$, the distribution of $M$ is given by
\begin{equation}\label{eq:distrib_rna}
\Prob_M(k) = \frac{\nu^k}{k!} \frac{(a)_k}{(b)_k} \hyper{N}{N}[(a+k);(b+k);-\nu] .
\end{equation}
\end{corollary}
Note that since matrices $-H$ and $-H_{\{1\}}$ are real, their complex eigenvalues come by conjugate pairs so $(a+k_1)_{k_2}$ and $(b+k_1)_{k_2}$ for $k_1,k_2\in\N$ are always real numbers.

\begin{remark}
When $\nu$ is large, the Poisson layer becomes negligible, meaning $M \approx \nu X_1$, so we can alternatively use~\eqref{eq:distrib_rna} to approximate the distribution of $X_1$. More precisely, if $f_{X_1}$ denotes the density of $X_1$, we have
\[\nu\gg 1 \quad\Rightarrow\quad f_{X_1}\left(k/\nu\right) \approx \Prob_M(k) \]
which in fact corresponds to the Post--Widder inversion formula applied to $L_{X_1}$.
\end{remark}

\subsection{Mixing distribution}

When computing $\Prob_M$, it is common for tractability reasons to take a rather small value for the scale parameter $\nu$, which is coherent since $\Prob_M(k)$ vanishes quickly for $k>\nu$ by definition of the Poisson mixture. However, quantitative biological experiments often suggest $\nu=10^3$ or more~\cite{Albayrak2016,Richard2016}.
Equation~\eqref{eq:distrib_rna} then corresponds as noted above to the Post--Widder inversion formula for the distribution of $X_1$, which emerges even more as the core part of $\Prob_M$.

It is therefore interesting to consider deriving the density $f_{X_1}$ in exact form, that is, directly inverting the Laplace transform~\eqref{eq:laplace_marginal}. Fortunately, one does not have to do this from scratch as $L_{X_1}$ belongs to a well-known class~\cite{Slater1966}.
More precisely, the idea is to invert the \emph{Mellin transform} of $X_1$, defined as the meromorphic function $M_{X_1}(z) = \Esp\left[X_1^{z-1}\right]$, which coincides with the moments of $X_1$ given from~\eqref{eq:laplace_marginal} by
\begin{equation}\label{eq:moments}
\Esp\left[X_1^k\right] = \frac{(a)_k}{(b)_k} = \prod_{i=1}^N \frac{\Gamma(b_i)\Gamma(a_i+k)}{\Gamma(a_i)\Gamma(b_i+k)} , \quad \forall k\in\N .
\end{equation}
It is possible to show using an extension of Carlson's theorem that the right-hand side of~\eqref{eq:moments} actually defines $M_{X_1}$ (replacing $k$ with $z-1$), but here such a technical result is not needed since we know by~\eqref{eq:series_moments} that the distribution of $X_1$ is characterized by its moments. Namely, we need only find the unique distribution on $[0,1]$ with moments~\eqref{eq:moments}, and by Mellin inversion this one is defined by the density
\begin{equation}\label{eq:density_mellin}
f_{X_1}(x) = \frac{1}{2\pi \mathrm{i}} \prod_{i=1}^N \frac{\Gamma(b_i)}{\Gamma(a_i)} \int_{0-\mathrm{i}\infty}^{0+\mathrm{i}\infty} \prod_{i=1}^N \frac{\Gamma(a_i+z)}{\Gamma(b_i+z)} x^{-z-1} \intd{z}
\end{equation}
for $x\in(0,1)$. Up to a multiplicative constant, this is a standard Meijer G-function~\cite{Springer1970} and thus one can efficiently compute $f_{X_1}$ using numerical packages such as \texttt{mpmath}~\cite{mpmath}. Furthermore, the following result provides an actual explicit form in most cases.

\begin{theorem}\label{theo:beta_product}
Assume that $a_i - a_j \notin \Z$ for all $i,j\in\lb 1,N \rb$, $i\neq j$. Then we have
\begin{equation}\label{eq:density_hyper}
f_{X_1}(x) = \frac{1}{A} \sum_{i=1}^{N} B_i \, x^{a_i - 1} \hyper{N}{N-1}[(1 + a_i - b);(1 - a_i');x]
\end{equation}
where $a_i' = (a_1,\dots,a_{i-1},a_{i+1},\dots,a_N) \in \C^{N-1}$,
\[A = \prod_{i=1}^{N}\frac{\Gamma(a_i)}{\Gamma(b_i)} \quad \text{and} \quad B_i = \frac{\prod_{j\neq i} \Gamma(a_j - a_i)}{\prod_{j=1}^N \Gamma(b_j - a_i)} .\]
\end{theorem}
The proof simply consists in evaluating~\eqref{eq:density_mellin} by the method of residues; it appears in~\cite[p. 152, equation (4.8.1.16)]{Slater1966} as a particular case. Note that the poles of the integrand (located at $z = -a_i-k$ for $k\in\N$) are simple by hypothesis. The case with multiple poles is treated extensively in~\cite{Springer1970} but is more involved.
When $a_i-b_j\in\N$, one simply has $B_i = 0$ so there is no restriction on $b$. 
Note also that \texttt{mpmath} appears to use~\eqref{eq:density_hyper} with a perturbation technique in the general case rather than performing complex integration in~\eqref{eq:density_mellin}.

\subsection{Beta-product distributions}

Consider the case $n=3$ with rates $r_{1,2}=10$, $r_{2,3}=r_{3,1}=2$, $r_{3,2}=1$ and $r_{2,1}=r_{1,3}=0$. This gives $a = (1,4)$ and $b = (6,9)$.
Then it is easy to show using moments~\eqref{eq:moments} that $X_1$ has the same distribution as $Z_1Z_2$ where $(Z_1,Z_2) \sim \Bet(1,5)\otimes\Bet(4,5)$ or equivalently $(Z_1,Z_2) \sim \Bet(1,8)\otimes\Bet(4,2)$.
More generally, if $a$ and $b$ are real with $a_i < b_i$ for all $i$, then $X_1$ can be interpreted as a product of independent Beta-distributed random variables~\cite{Springer1970}, namely, $X_1 = Z_1 \cdots Z_N$ with $Z_i \sim \Bet(a_i,b_i-a_i)$.
These Beta-product distributions notably arise in different contexts such as statistics~\cite{Dufresne2010} and mathematical physics~\cite{Dunkl2013} and it is fruitful to extend them to complex $a$ and $b$ such that the moment sequence~\eqref{eq:moments} stays real: this indeed generates “new” distributions that cannot be realized as Beta products with real parameters~\cite{Dufresne2010}.
From this viewpoint, the multivariate PDMP turns out to be a very natural way to generate, using real transition matrices, Beta-product distributions with complex parameters.

In the case $n=3$, we mention an alternative to~\cref{theo:beta_product} that is always valid (i.e., also when $a_1 - a_2 \in \Z$). Indeed, similarly to the real case~\cite{Dufresne2010,Dunkl2013} we get
\[f_{X_1}(x) = \frac{\Gamma(b_1)\Gamma(b_2)}{\Gamma(a_1)\Gamma(a_2)\Gamma(\delta)} x^{a_1-1}(1-x)^{\delta-1} {}_2F_1[(b_1-a_2,b_2-a_2);(\delta);1-x] \]
where $\delta = b_1+b_2-a_1-a_2 = r_{1,2}+r_{1,3} >0$.
Note that $a_1$ and $a_2$ are always real in this case so the hypergeometric series has nonnegative coefficients, and thus $f_{X_1}$ can be interpreted as a mixture of Beta distributions.

\section{Applications}
\label{sec:applications}

In this section we show some interesting examples of refractory promoters, still assuming that the active state is $i=1$ and taking $u_1=\nu=10^3$. The general formulas for distributions of $T_0$, $M$ and $X_1$ were implemented in Python, making use of the \texttt{mpmath} package~\cite{mpmath} to compute~\eqref{eq:distrib_rna} and~\eqref{eq:density_mellin}. The code thereby covers all refractory promoters (i.e., with any number of states and any transition graph) and is available at \url{https://github.com/ulysseherbach/multistate}.

\subsection{Irreversible cyclic promoters}

In this case, the promoter is progressing irreversibly through $N$ inactive states (from $2$ to $n=N+1$) before reaching the active state $1$, and so on (\cref{fig:refractory}). It is a straightforward generalization~\cite{Zoller2015} of the two-state model, which corresponds to $N=1$.
As shown in~\cref{fig:irreversible_cycle}, increasing $N$ while keeping $\Esp[T_0]$ fixed tends to decrease $\Var(T_0)$, which rather intuitively decreases $\Var(M)$.

\begin{figure}[htbp]
\begin{center}
\includegraphics[width=\textwidth]{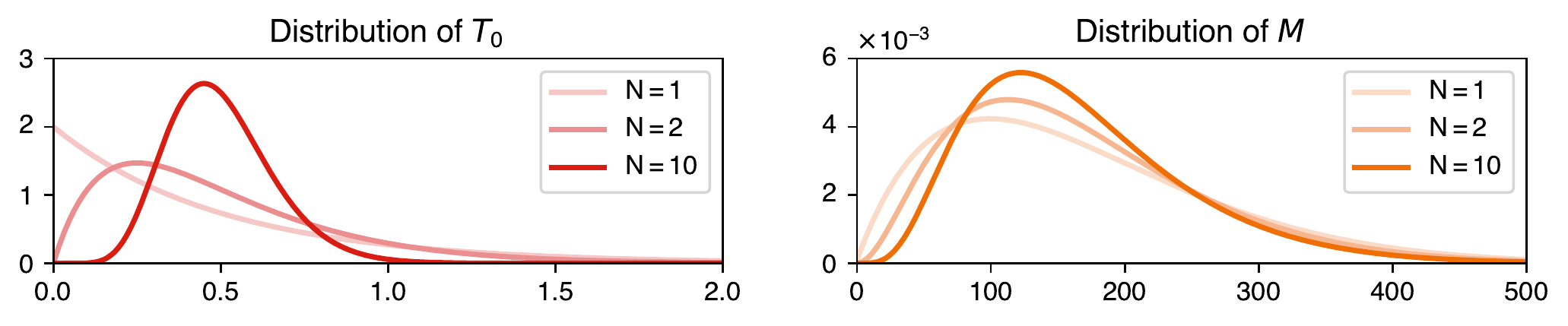} \vspace{-6mm}
\caption{Irreversible cyclic refractory promoters for $n = N+1$ with $N\geqslant 1$ inactive states. Here we set $r_{1,2}=10$ and $r_{i,i+1}=r_{n,1}=2N$ for $i\in \lb 2,N \rb$, other rates being zero, so $T_0 \sim \Gam(N,2N)$. For $N=10$, the distribution of $M$ is close to its limit when $N\to\infty$, corresponding to $T_0 = 1/2$.}
\label{fig:irreversible_cycle}
\end{center}
\end{figure}

\subsection{Irreversible cycle with a shortcut}

We now consider a more complex inactive period (\cref{fig:shortcut_promoter}), characterized by a five-state cycle with a “shortcut” from state $1$ to state $4$. The rates are chosen so that the promoter follows the long cycle most of the time, whereas it can sometimes bypass states $2$ and $3$, leading to a bimodal distribution for $T_0$ (available by \cref{lem:inactive_period}). However, the two modes are not easily detectable in sample paths and the result is indeed a unimodal distribution for $M$.

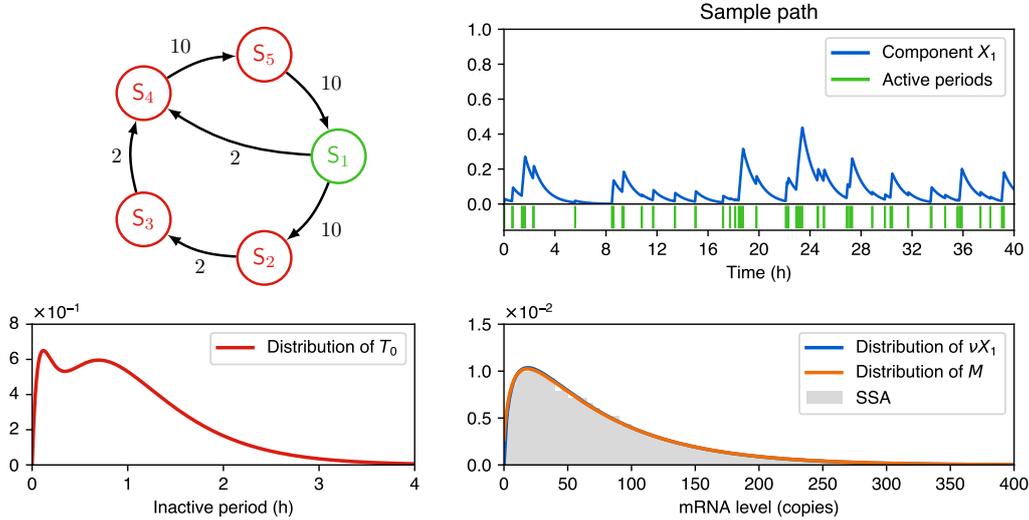
\begin{figure}[htbp]
\begin{center}
\resizebox{\textwidth}{!}{\begin{tikzpicture}
\node at (0,0) {\includegraphics[width=\textwidth]{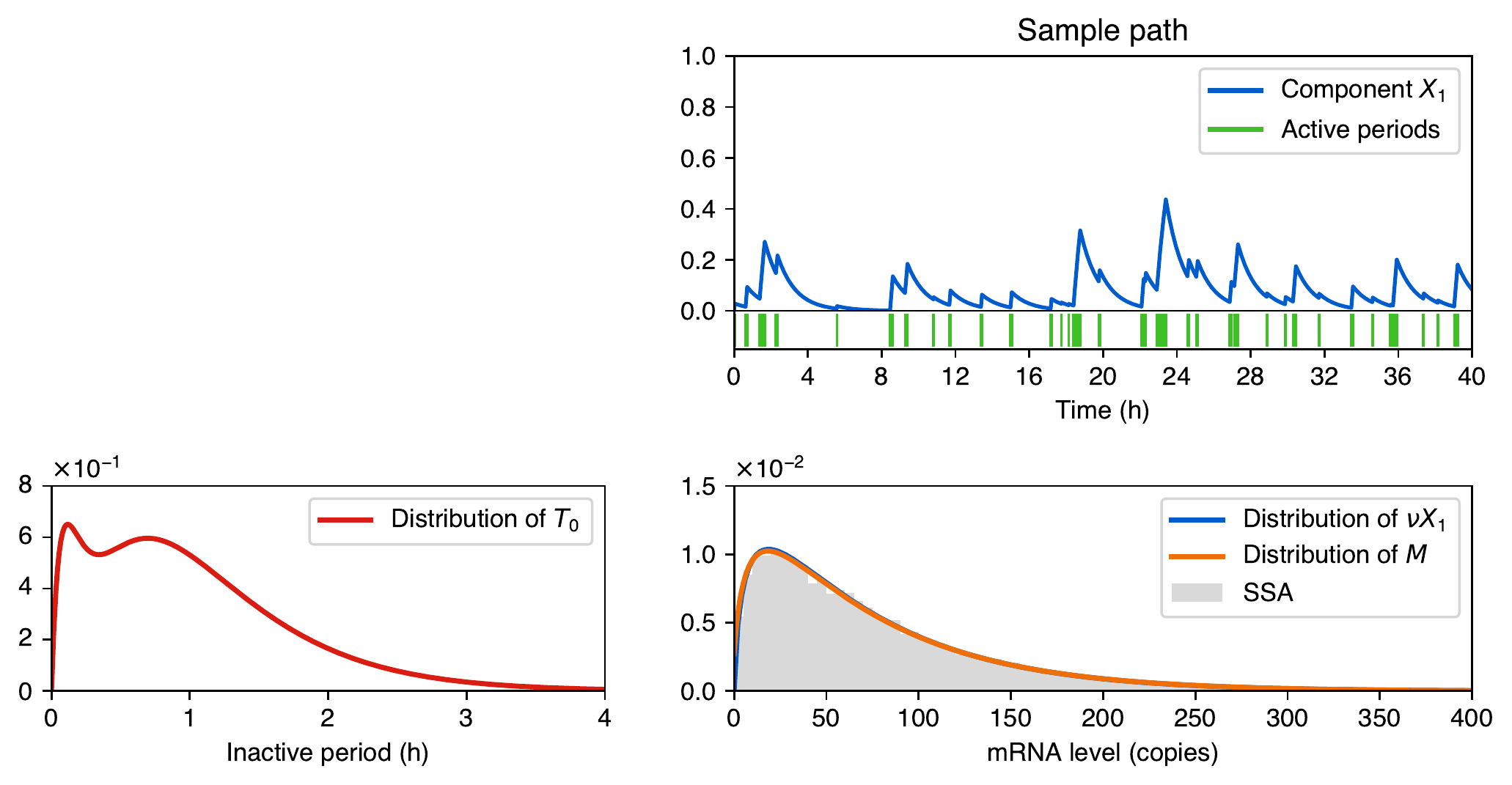}};
\node at (-3.75,1.4) {\resizebox{!}{3.5cm}{%
\begin{tikzpicture}[scale=0.9, line width = 1.2]
\tikzstyle{active}=[draw, circle, minimum width=9mm, color=vert]
\tikzstyle{inactive}=[draw, circle, minimum width=9mm, color=rouge]
\tikzstyle{bleft}=[bend left=15]
\tikzstyle{arrow}=[{->},>=latex]
\node[active]   (S1) at (2.00,-0.00) {$\chem{S}_1$};
\node[inactive] (S2) at (0.62,-1.90) {$\chem{S}_2$};
\node[inactive] (S3) at (-1.62,-1.18) {$\chem{S}_3$};
\node[inactive] (S4) at (-1.62,1.18) {$\chem{S}_4$};
\node[inactive] (S5) at (0.62,1.90) {$\chem{S}_5$};
\draw[arrow] (S1) to[bleft] node[midway, below right]{\small $10$} (S2);
\draw[arrow] (S1) to[bleft] node[midway, below]{\small $2$} (S4);
\draw[arrow] (S2) to[bleft] node[midway, color=black, below]{\small $2$} (S3);
\draw[arrow] (S3) to[bleft] node[midway, left]{\small $2$} (S4);
\draw[arrow] (S4) to[bleft] node[midway, above left]{\small $10$} (S5);
\draw[arrow] (S5) to[bleft] node[midway, above right]{\small $10$} (S1);
\end{tikzpicture}
}};
\end{tikzpicture}} \vspace{-6mm}
\caption{Example of refractory promoter with a bimodal inactive period. Left: here $n=5$ and we set $r_{1,4}=r_{2,3}=r_{3,4}=2$ and $r_{1,2}=r_{4,5}=r_{5,1}=10$, other rates being zero. Right: sample path of the component $X_1$ of the multivariate PDMP and exact stationary distributions of $M$ (discrete) and $\nu X_1$ (continuous). The two distributions turn out to be very similar since $\nu = 10^3$ and their formulas are confirmed by sampling $M$ ($10^4$ cells) using the stochastic simulation algorithm (SSA).}
\label{fig:shortcut_promoter}
\end{center}
\end{figure}

\subsection{Multiple cycles}

Finally,~\cref{fig:two_cycle_promoter} shows a promoter with two distinct cycles, which leads to a bimodal distribution for $M$. This time one can see two typical inactive periods in sample paths, but $T_0$ appears unimodal: in fact the long cycle is rare compared to the short one so the corresponding mass is flattened.

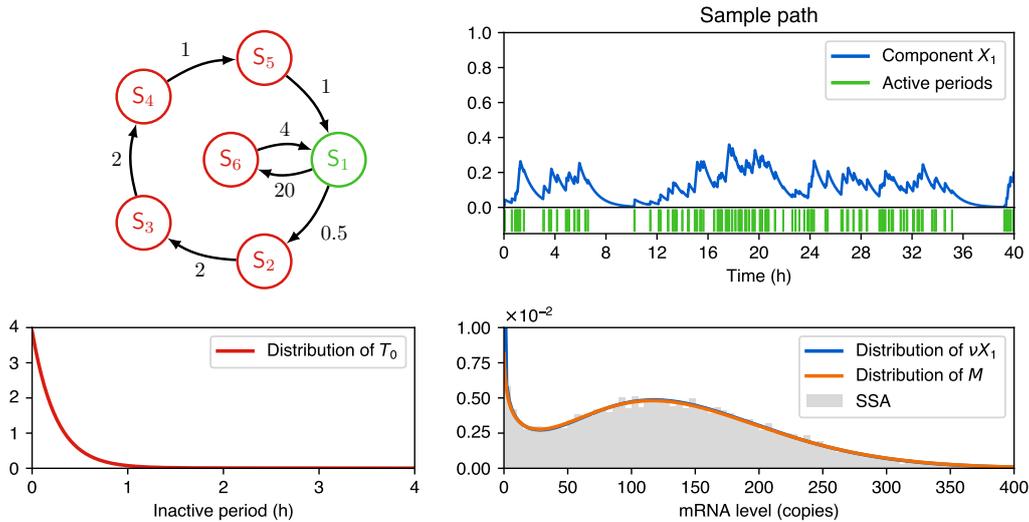
\begin{figure}[htbp]
\begin{center}
\resizebox{\textwidth}{!}{\begin{tikzpicture}
\node at (0,0) {\includegraphics[width=\textwidth]{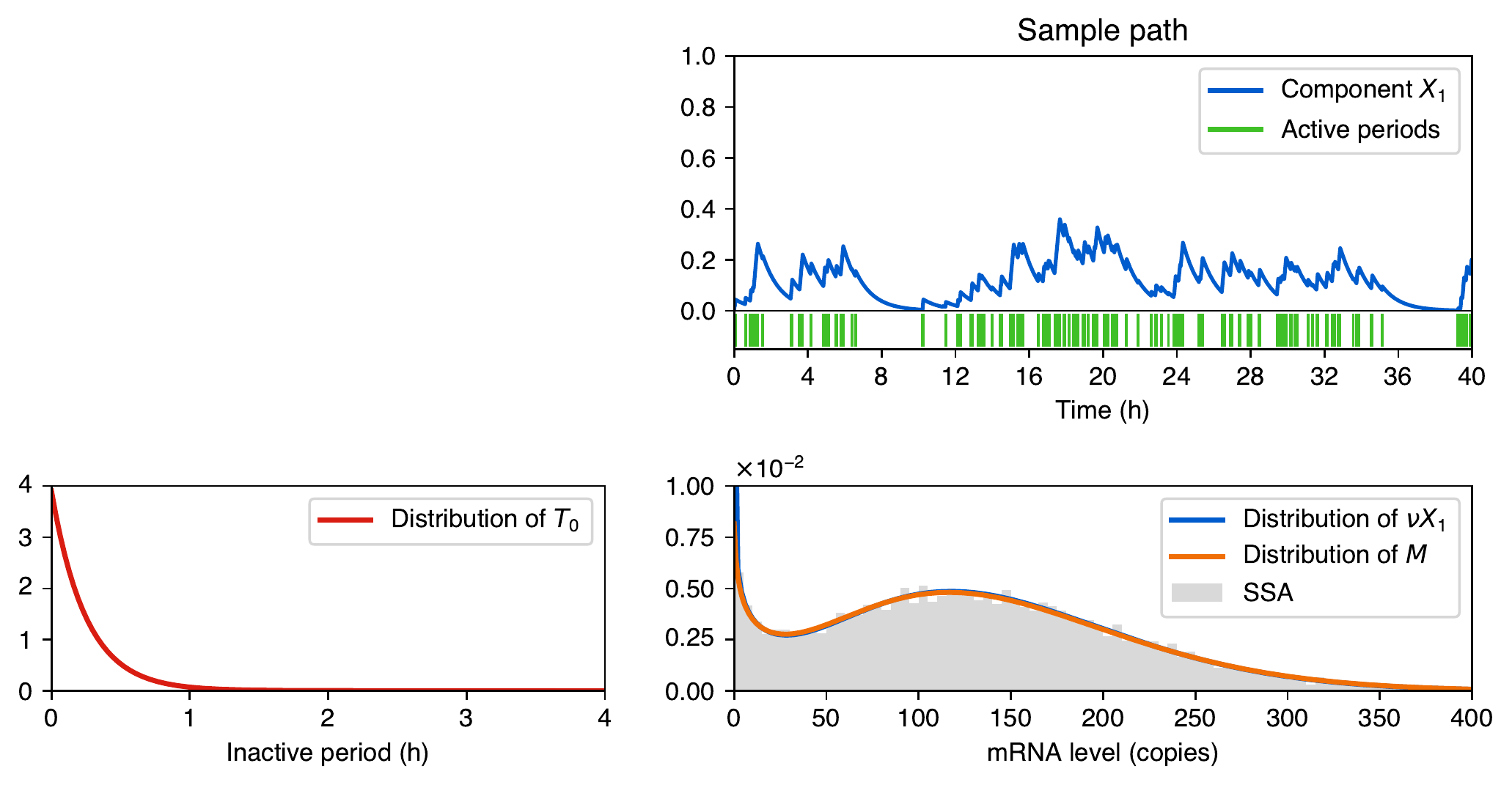}};
\node at (-3.75,1.4) {\resizebox{!}{3.5cm}{%
\begin{tikzpicture}[scale=0.9, line width = 1.2]
\tikzstyle{active}=[draw, circle, minimum width=9mm, color=vert]
\tikzstyle{inactive}=[draw, circle, minimum width=9mm, color=rouge]
\tikzstyle{bleft}=[bend left=15]
\tikzstyle{arrow}=[{->},>=latex]
\node[active]   (S1) at (2.00,-0.00) {$\chem{S}_1$};
\node[inactive] (S2) at (0.62,-1.90) {$\chem{S}_2$};
\node[inactive] (S3) at (-1.62,-1.18) {$\chem{S}_3$};
\node[inactive] (S4) at (-1.62,1.18) {$\chem{S}_4$};
\node[inactive] (S5) at (0.62,1.90) {$\chem{S}_5$};
\node[inactive] (S6) at (0,0) {$\chem{S}_6$};
\draw[arrow] (S1) to[bleft] node[midway, below right]{\small $0.5$} (S2);
\draw[arrow] (S2) to[bleft] node[midway, color=black, below]{\small $2$} (S3);
\draw[arrow] (S3) to[bleft] node[midway, left]{\small $2$} (S4);
\draw[arrow] (S4) to[bleft] node[midway, above left]{\small $1$} (S5);
\draw[arrow] (S5) to[bleft] node[midway, above right]{\small $1$} (S1);
\draw[arrow] (S1) to[bend left=20] node[midway, below]{\small $20$} (S6);
\draw[arrow] (S6) to[bend left=20] node[midway, above]{\small $4$} (S1);
\end{tikzpicture}
}};
\end{tikzpicture}} \vspace{-6mm}
\caption{Example of refractory promoter leading to a bimodal mRNA distribution. Left: here $n=6$ and we set $r_{4,5}=r_{5,1}=1$, $r_{2,3}=r_{3,4}=2$, $r_{1,2}=0.5$, $r_{1,6}=20$ and $r_{6,1}=4$, other rates being zero. Right: sample path of the component $X_1$ of the multivariate PDMP and exact stationary distributions of $M$ (discrete) and $\nu X_1$ (continuous). As in \cref{fig:shortcut_promoter}, the two distributions turn out to be very similar since $\nu = 10^3$ and their formulas are confirmed by sampling $M$ ($10^4$ cells) using the stochastic simulation algorithm (SSA).}
\label{fig:two_cycle_promoter}
\end{center}
\end{figure}

\section{Discussion and perspectives}
\label{sec:discussion}

In this paper we derived an explicit Poisson representation for a standard model of gene expression, based on a multistate promoter. The explicit form of the mRNA distribution in the refractory case is similar to that of~\cite{Dattani2016,Zhou2012}, but we used the underlying linear structure rather that exploiting particular promoter transitions. This led to more general results with simpler proofs, and also enabled us to identify marginals of the underlying multivariate PDMP as extending the class of beta-product distributions.

Compared to the original~\cite{Gardiner1977}, our approach is hybrid in that we only represent $\chem{M}$ by a Poisson mixture while keeping the common description of $\chem{S}_i$. This is crucial as the underlying dynamics then correspond to a well-defined Markov process, which is not the case in general: when some reactions involve or produce more than one molecule, the Poisson representation does not lead to a pure transport equation anymore---there can be diffusion or higher-order noise terms---and may not even exist as a real-valued Poisson mixture (see \cref{rem:complex_measure}).
One can check from \eqref{eq:chem_multistate} that the only problematic reaction here is $\chem{S}_i \xrightarrow[]{} \chem{S}_i + \chem{M}$. By conditioning on the promoter state $\chem{S}_i$, we actually get a monomolecular reaction $\varnothing \xrightarrow[]{} \chem{M}$ with rate $0$ or $u_i$ and therefore obtain a purely deterministic behavior between jumps of $[\chem{S}_i]\in\{0,1\}$.

The PDMP viewpoint is itself becoming well established in biological applications because of its great [modeling power]/[mathematical complexity] ratio~\cite{Rudnicki2017}. In fact, it is relevant and already used in various situations outside biology, for example in the so-called fluid queuing theory where the two-state model also has a meaning~\cite{Boxma2005}.
\Cref{fig:shortcut_promoter,fig:two_cycle_promoter} show that in biologically realistic conditions~\cite{Albayrak2016,Singer2014}, the distribution of $M$ efficiently approximates the one of $\nu X_1$, or in other words the Poisson layer adds a very small amount of noise to the PMDP layer.
Besides, the mRNA bimodality in~\cref{fig:two_cycle_promoter} is exactly the one observed in practice~\cite{Singer2014}, that is, with a gamma-like tail. We emphasize that such multimodality differs from the one considered in~\cite{Innocentini2013}, which comes from long stays in distinct active states (i.e., with different $u_i$ values) and has a much shorter, Poissonian tail (see~\cite{Albayrak2016} for a quantitative illustration).
In particular, contrary to a somewhat widespread belief, the two-state model is absolutely unable to reproduce gamma-tailed bimodality. The promoter of~\cref{fig:two_cycle_promoter} gets around this in a rather ad hoc way by generating two latent “bursting frequencies” corresponding to two cycles with very different mean durations. It might be better to let these frequencies emerge from actual gene networks such as the two-gene toggle switch, unless one has strong biological evidence for a particular promoter structure.

One can also obtain nontrivial multimodality by considering a single gene with feedback, i.e., with promoter transition rates now being functions of the mRNA quantity. More specifically, two-state models with linear or second-order polynomial feedback functions have been successfully investigated~\cite{Hornos2005,Iyer-Biswas2014,Dattani2016}. Interestingly, while second-order feedback is hard to tackle analytically, the case of linear feedback can be easily incorporated into the Poisson representation, for the same monomolecular aspect as above. Note that the underlying stochastic process is well defined but is actually not a PDMP because of some particular advection terms that appear~\cite{Dattani2016,Iyer-Biswas2014}.
Note also that gamma-tailed bimodality cannot be reproduced using unbounded feedback functions such as polynomials: to obtain genuine bursting frequencies, any positive feedback function should rather be bounded (e.g., of Hill type) so the activation rate stays much smaller than the inactivation rate. This makes the Poisson representation likely intractable, but the analysis can be made possible and still biologically relevant by directly starting from the PDMP formulation~\cite{Herbach2017}.

Having \cref{prop:dirichlet_promoter} in mind, it is clear that vectors $a$ and $b$ are in general not identifiable from the mRNA distribution, as Dirichlet marginals are Beta and thus indistinguishable from the two-state model.
In practice,~\cref{fig:irreversible_cycle,fig:shortcut_promoter,fig:two_cycle_promoter} suggest that distributions of $M$ or $\nu X_1$ in the bursty regime may be reasonably approximated by gamma or mixtures of gamma distributions. This favors the two-state model as a relevant approximation in many cases since the gamma distribution is nothing but the bursty limit (i.e., $r_{1,2} \gg r_{2,1}$ if the active state is $i=1$) of this model~\cite{Albayrak2016,Herbach2017}.
Furthermore, \cref{fig:shortcut_promoter,fig:two_cycle_promoter} illustrate some typical issues that would appear when trying to infer a complex promoter structure from experimental data without additional knowledge: \cref{fig:shortcut_promoter} indeed shows that mRNA data is not sufficient to reconstruct the structure, and---contrary to what could have been expected---\cref{fig:two_cycle_promoter} shows that measuring the distribution of the inactive period $T_0$ is not enough either unless measurements are extremely precise.

Finally, we were only able to get the joint distribution of the multivariate PDMP in a particular case where it turns out to be the Dirichlet distribution (see \cref{sec:dirichlet}). The joint distribution seems much more involved in general, but it may hopefully be a natural extension of the latter distribution. Intuitively, the difficulty comes from the dependence between components $X_1,\dots,X_n$ which is reduced to $X_1+\cdots+X_n=1$ in the Dirichlet case.
Knowing the joint distribution would be interesting not only mathematically, but also from a biological point of view as it would enable us to describe further complexity layers. Indeed, the translation stage is commonly modeled by
\[\chem{M} \xrightarrow[]{v} \chem{M} + \chem{P}, \quad  \chem{P} \xrightarrow[]{d_1} \varnothing\]
where $\chem{P}$ is the translated protein, and this stage can clearly be viewed for $\chem{P}$ as the transcription stage with respect to $\chem{M}$. Hence, the multivariate PDMP approach could hopefully give some useful insight for deriving the exact protein distribution, which is known to be a very difficult problem.

\appendix

\section{Spaces}
\label{sec:spaces}%
In this section we provide some details about function spaces and operators used in this paper.
First, $\mathcal{M}(\R_+)$ and $\mathcal{M}(\N)$ respectively denote measures on $(\R_+,\mathcal{B}(\R_+))$ and $(\N,\mathcal{P}(\N))$, where $\mathcal{B}(\R_+)$ is the Borel algebra over $\R_+$ and $\mathcal{P}(\N)$ is the power set of $\N$.
The first space is equipped with the total variation norm, defined by
\[\|\mu\| = \sup\{\mu(A) - \mu(\R_+\!\!\setminus\! A) \;|\; A\in \mathcal{B}(\R_+)\} = \mu^+(\R_+) + \mu^-(\R_+)\]
where $\mu = \mu^+ - \mu^-$ is the Jordan decomposition of $\mu\in\mathcal{M}(\R_+)$.
The total variation norm is defined similarly on $\mathcal{M}(\N)$, the Jordan decomposition being trivial in this case.
Furthermore, the Riesz--Markov representation theorem identifies $\mathcal{M}(\R_+)$ and $\mathcal{M}(\N)$ as the duals of Banach spaces $(\mathcal{C}_0(\R_+),\|\cdot\|_\infty)$ and $(c_0(\N),\|\cdot\|_\infty)$, respectively, where $\mathcal{C}_0(\R_+)$ is the space of continuous functions on $\R_+$ converging to zero at $+\infty$ and $c_0(\N)$ is the space of real sequences converging to zero.
Remarkably, $\|\cdot\|$ then corresponds to the dual norm so $(\mathcal{M}(\R_+),\|\cdot\|)$ and $(\mathcal{M}(\N),\|\cdot\|)$ are Banach spaces.
Finally, all of this extends to $\mathcal{M}_n(\R_+)$ and $\mathcal{M}_n(\N)$ as defined in~\cref{seq:poisson:notation} (say, with $\|\mu\| = \|\mu_1\| + \cdots + \|\mu_n\|$ for $\mu\in\mathcal{M}_n(\R_+)$ or $\mathcal{M}_n(\N)$), which are duals of
\[\mathcal{C}_0^n = \mathcal{C}_0(\lb 1,n \rb \times \R_+) = \mathcal{C}_0(\R_+)^n \quad \text{and} \quad c_0^n = c_0(\lb 1,n \rb \times \N) = c_0(\N)^n .\]

We now introduce the generator $\mathcal{L}$ of the jump Markov process $(E_t,M_t)_{t\geqslant 0}$ as the infinitesimal generator of the semigroup of bounded operators $T_t : c_0^n \to c_0^n$ defined by
\[T_t f(i,k) = \Esp[f(E_t,M_t)|E_0=i,M_0=k] , \quad \forall (i,k)\in\lb 1,n \rb \times \N\]
for $t\geqslant 0$ and $f\in c_0^n$.
Similarly, the generator $\widetilde{\mathcal{L}}$ of the piecewise-deterministic Markov process $(E_t,Y_t)_{t\geqslant 0}$ is the infinitesimal generator of $\widetilde{T}_t : \mathcal{C}_0^n \to \mathcal{C}_0^n$ defined by
\[\widetilde{T}_t f(i,x) = \Esp[f(E_t,Y_t)|E_0=i,Y_0=x] , \quad \forall (i,x)\in\lb 1,n \rb \times \R_+\]
for $t\geqslant 0$ and $f\in \mathcal{C}_0^n$.
Note that $(T_t)_{t\geqslant 0}$ and $(\widetilde{T}_t)_{t\geqslant 0}$ are indeed (strongly continuous) semigroups and that $\mathcal{L}$ and $\widetilde{\mathcal{L}}$ coincide with~\eqref{eq:full_gen_backward} and~\eqref{eq:pdmp_gen_backward}: this is standard and follows from construction of the processes. Besides, we do not need the precise domains of the generators, but only subspaces that are dense in $(c_0^n,\|\cdot\|_\infty)$ and $(\mathcal{C}_0^n,\|\cdot\|_\infty)$: one can choose sequences that have finitely many nonzero elements and restrictions to $\R_+$ of compactly supported smooth functions on $\R$, respectively.

As a result, the standard semigroup theory directly applies and the Kolmogorov backward equations of $(E_t,M_t)_{t\geqslant 0}$ and $(E_t,Y_t)_{t\geqslant 0}$ can be defined as well-posed Cauchy problems~\cite{Pazy1983}. Concerning the forward equations~\eqref{eq:full_master} and~\eqref{eq:pdmp_master}, which are the main point of~\cref{prop:stable_space}, our choice here is to directly consider the well-defined adjoint semigroups $S_t = T_t^*$ and $\widetilde{S}_t = \widetilde{T}_t^*$ rather than attempting at a precise definition of the forward Cauchy problems, since their solutions should be based on these adjoint semigroups anyway.

\begin{remark}
The semigroup $S_t : \mathcal{M}_n(\N) \to \mathcal{M}_n(\N)$ is indeed strongly continuous with generator $\Omega$ as in~\eqref{eq:full_gen_forward} but $\widetilde{S}_t : \mathcal{M}_n(\R_+) \to \mathcal{M}_n(\R_+)$ is \emph{not}: the domain of its generator $\widetilde{\Omega}$ is not dense in $(\mathcal{M}_n(\R_+),\|\cdot\|)$, so there is no hope for a dense subspace on which it could be defined “strongly”. To avoid this, a typical option is to embed $\mathcal{M}_n(\R_+)$ into a larger space of generalized functions and define $\widetilde{\Omega}$ in a weak sense, as done implicitly in~\eqref{eq:pdmp_gen_forward}. It is also possible~\cite{Rudnicki2017} to consider the subspace $L^1(\R_+)^n$, on which $\|\cdot\|$ coincides with $\|\cdot\|_1$ so that $\widetilde{\Omega}$ can be densely defined on smooth functions. Alternatively, we conjecture that one may get a strongly continuous semigroup using the Kantorovich--Rubinstein norm (see~\cite{Hille2009} and references therein). This is beyond the scope of this article but would be a very natural choice for applying the (forward) semigroup theory to PDMPs in general (see~\cite{Benaim2012}).
\end{remark}

\section{Deriving the Poisson representation}
This section provides some details on the different approaches toward the Poisson representation.

\subsection{Ansatz-based approach}
\label{sec:ansatz:approach}

This is the direct but nonrigorous way to derive the operator $\widetilde{\Omega}$.
Let $f=(f_1,\dots,f_n)^\top$ be a smooth density function on $\lb 1,n \rb \times \R_+$ and consider $p = Pf$. Let us express $\Omega p$ from definition~\eqref{eq:full_gen_forward} in terms of $f$:
\begin{enumerate}
\item \emph{Degradation}. By integration by parts, we have
\[\forall k \geqslant 1, \quad (k+1)p(k+1) - kp(k) = \int_0^{\infty} \frac{y^k}{k!}e^{-y} (\dr_y(yf)) \intd{y} = P[\dr_y(yf)](k)\]
whenever boundary terms vanish: a sufficient condition is
\[\lim_{y\to 0}(yf(y)) = \lim_{y\to \infty}(e^{-y}yf(y)) = 0 .\]
\item \emph{Creation}. Using the same boundary condition on $f$, we obtain
\[\forall k \geqslant 1, \quad C[p(k-1) - p(k)] = \int_0^{\infty} \frac{y^k}{k!}e^{-y} (-C\dr_y f) \intd{y} = P[-C\dr_y f](k)\]
\item \emph{Promoter transitions.} Since $H$ does not depend on mRNA, we directly get
\[\forall k\geqslant 0, \quad Hp(k) = H[Pf](k) = P[Hf](k) .\]
\end{enumerate}
We obtain~\eqref{eq:pdmp_master} by gathering the three pieces, forgetting the problem with $k=0$ in the first two, and recalling that $P$ is injective.

\subsection{Dual approach}
\label{sec:dual:approach}

Here we give some details on~\cref{prop:stable_space}. Recall that by~\cref{sec:spaces}, we got well-defined semigroups $S_t = T_t^*$ and $\widetilde{S_t} = \widetilde{T}_t^*$.
Thanks to the dual approach, one can derive~\eqref{eq:pde_generating} and~\eqref{eq:pde_laplace} by applying $T_t$ to functions
\[f_z : (i,k)\mapsto (\mathds{1}_{i=1} z^k, \dots, \mathds{1}_{i=n} z^k)^\top \in c_0^n\]
for $z\in(-1,1)$, and applying $\widetilde{T}_t$ to functions
\[f_s : (i,y)\mapsto (\mathds{1}_{i=1} e^{sy}, \dots, \mathds{1}_{i=n} e^{sy})^\top \in \mathcal{C}_0^n\]
for $s\in(-\infty,0)$, and using definitions of $\mathcal{L}$ and $\widetilde{\mathcal{L}}$.
Let us now derive~\cref{prop:stable_space} as a direct consequence.
Let $\mu_0\in \mathcal{M}_n(\R_+)$, $p_0 = P\mu_0$, $p(t) = S_t p_0$ and $\mu(t) = \widetilde{S}_t\mu_0$.
Then we have $G_{p(0)}(z) = L_{\mu(0)}(z-1)$, and it follows that $G_{p(t)}(z) = L_{\mu(t)}(z-1)$ for all $t\geqslant 0$ thanks to~\eqref{eq:pde_generating}-\eqref{eq:pde_laplace}.
Thus, by~\cref{lem:feller}, $p(t)=P\mu(t)$ for all $t\geqslant 0$, which can be written $S_t P\mu_0 = P\widetilde{S}_t\mu_0$ or equivalently $P^{-1}S_t P\mu_0 = \widetilde{S}_t\mu_0$.

\subsection{Path-based approach}
\label{sec:path:approach}

Given $(E_t)_{t\geqslant0}$, the conditional master equation is
\begin{equation*}
\dr_t \overline{p}(k,t) = d_0[(k+1)\overline{p}(k+1,t) - k\overline{p}(k,t)] + \overline{u}(E_t)[\overline{p}(k-1,t)\mathds{1}_{k>0} - \overline{p}(k,t)]
\end{equation*}
where $\overline{p}(k,t) = \Prob(M_t=k | (E_\tau)_{\tau\geqslant 0}) = \Prob(M_t=k | (E_\tau)_{\tau\in [0,t]})$ and with $\overline{u}$ as in~\eqref{eq:ode_univariate}.
Note that $\overline{p}$ can give back the solution $p$ of the original master equation~\eqref{eq:full_master} when integrated appropriately.

\section{Spectral properties of promoter transitions}
\label{sec:spectral_properties}%
From \cref{sec:model}, the transpose of the promoter transition matrix is defined by
\[H_{i,j} = r_{j,i} \quad \text{for} \;\; i\neq j \quad \text{and} \quad H_{i,i} = -\sum_{j\neq i} r_{i,j} .\]
In this section, we prove \cref{lem:eigenvalues} concerning the eigenvalues $a^i_1,\dots,a^i_N$ of “principal submatrices” $-H_{\{i\}}\in\Mat{N}$ for $i\in \segn$ where $n=N+1$.
Let us first recall how to derive the fact that $0$ is a simple eigenvalue of $-H$ with all its other eigenvalues $b_1,\dots,b_N$ having positive real parts. The main two ingredients are Gershgorin's circle theorem and the Perron--Frobenius theorem.

Let $\alpha = \max_i \big\{\sum_{j\neq i} r_{i,j} \big\}$. Since $H$ is irreducible, we have $\alpha > 0$ and the matrix
\[M = \frac{1}{\alpha}H + I\]
is irreducible and nonnegative.
Its spectral radius $\rho(M)$ satisfies $\rho(M)\geqslant 1$ since $1$ is clearly an eigenvalue of $M$, and it is straightforward to see by Gershgorin's circle theorem and by construction of $H$ that $\rho(M)\leqslant 1$, so $\rho(M)=1$. Hence $1$ is a simple eigenvalue of $M$ by the Perron--Frobenius theorem so $0$ is a simple eigenvalue of $H$. Finally, applying Gershgorin's circle theorem to $H$ shows that the only possible nonnegative eigenvalue of $H$ is 0, so $b_1,\dots,b_N$ all have positive real parts.

Second, consider submatrices $H_{\{i\}}\in\Mat{N}$. Again, Gershgorin's theorem shows that the only possible eigenvalue of $H_{\{i\}}$ with a nonnegative real part would be $0$. But by~\cite[Lemma 1]{Innocentini2013}, we know that the vector
\[v_0 = (\det(H_{\{1\}}), \dots, \det(H_{\{n\}}))^\top\]
is a Perron--Frobenius eigenvector of $H$ (i.e. associated with dominant eigenvalue $0$), meaning that all its components are nonzero. As a result, $0$ cannot be an eigenvalue of $H_{\{i\}}$ and $a^i_1,\dots,a^i_N$ all have positive real parts.

Finally, products $\prod_{k=1}^N b_k$ and $\prod_{k=1}^N a^i_k$ for $i\in \segn$ are real and positive since the related eigenvalues come by conjugate pairs, and standard results on the characteristic polynomial of $H$ show that
\[\prod_{k=1}^N b_k = \sum_{i=1}^n \prod_{k=1}^N a^i_k .\]
Thus, we get
\[\Prob_E(i) = \prod_{k=1}^N \frac{a^i_k}{b_k} > 0, \quad \forall i\in\lb 1,n \rb \]
using the probability condition and $\Prob_E \varpropto v_0$.

\phantomsection
\pdfbookmark[section]{Acknowledgments}{acknowledgments}
\section*{Acknowledgments}
The author is very grateful to Thibault Espinasse, Olivier Gandrillon and Anne-Laure Fougères for their constant support.
Together we thank the BioSyL Federation and the LabEx Ecofect (ANR-11-LABX-0048) of the University of Lyon for inspiring scientific events. The author also thanks the anonymous referees for their significant contribution to the improvement of the original manuscript.

\phantomsection
\pdfbookmark[section]{References}{references}
\bibliographystyle{siamplain}

\end{document}